\pdfoutput=1   

\documentclass[a4paper, 12pt]{article}
\usepackage[a4paper, hmargin= 2cm, vmargin=2cm]{geometry}

\usepackage{graphicx}
\usepackage{amsmath, amssymb, amsfonts, amscd, amsthm, mathrsfs}
\usepackage[all]{xy}
\usepackage[colorlinks, linkcolor= blue, citecolor= red]{hyperref}
\usepackage[capitalise]{cleveref}
\usepackage{pdfsync}
\usepackage{color}

\usepackage[T1]{fontenc}
\usepackage[utf8]{inputenc}

\usepackage{fbb}
\usepackage[libertine]{newtxmath}
\usepackage[cal=boondoxo,bb= pazo,frak=euler]{mathalfa}

\usepackage{titlesec}

\titleformat{\section}{\centering\sc}{\thesection .}{.5em}{}[]

\titleformat{\subsection}{\sl}{\thesubsection .}{.5em}{}[]

\newcommand{\stab}{\operatorname{Stab}}
\newcommand{\ttot}{T_{212^*}}
\renewcommand{\phi}{\varphi}
\newcommand{\C}{\mathbb{C}}

\newcommand{\n}{\mathbb{N}}

\newcommand{\f}{\mathbb{F}}

\newcommand{\A}{\mathcal{A}}
\newcommand{\W}{\mathcal{W}}
\renewcommand{\Vert}{\operatorname{Vert}}
\newcommand{\Edge}{\operatorname{Edge}}
\newcommand{\End}{\operatorname{End}}
\newcommand{\Aut}{\operatorname{Aut}}
\newcommand{\Out}{\operatorname{Out}}
\newcommand{\I}{\mathbb{I}}

\newcommand{\AP}{\mathbf{A}}
\newcommand{\CP}{\mathbf{C}}
\renewcommand{\P}{\mathbf{P}}
\newcommand{\chamber}{\mathcal{C}}
\newcommand{\Aff}{\mathcal{A}\!\mathcal{f}\!\mathcal{f}}
\renewcommand{\circle}{\mathcal{Cl}}

\newcommand{\supp}{\operatorname{supp}}
\newcommand{\B}{\mathcal{B}}
\newcommand{\cayley}{\operatorname{C}}
\newcommand{\IH}{\mathcal{I}\! \mathcal{H}}

\newcommand{\image}[2]{\includegraphics[width=#1\textwidth]{images/#2}}

\newcommand{\imageright}[3]{\noindent\begin{minipage}{.33\textwidth}
#2
\end{minipage}~\begin{minipage}{.66\textwidth}
\hfill\image{#1}{#3}
\end{minipage}
}

\usepackage{afterpage}

\newenvironment{verticallycentered}
               {\topskip0pt\vspace*{\fill}}
               {\vspace*{\fill}\clearpage}%

\newcommand{\topnextpage}[1]{ \afterpage{#1} \bigskip}

\usepackage{tikz}
\usetikzlibrary{arrows}

\tikzstyle{sommet}=[circle,draw, scale=0.5]
\tikzstyle{infosommet}=[scale=0.75]
\tikzstyle{infoarete}=[midway, above, scale=0.85]

\newtheoremstyle{pedro}{}{}{\itshape}{}{\bfseries}{.}{ }{\thmname{#1}\thmnumber{ #2}\thmnote{ (#3)}}

\newtheoremstyle{pedrodef}{}{}{}{}{\bfseries}{.}{ }{\thmname{#1}\thmnumber{ #2}\thmnote{ (#3)}}

\newtheoremstyle{pedroimage}{}{}{\itshape\footnotesize}{}{\itshape\footnotesize}{.}{ }{\thmname{#1}\thmnumber{ #2}\thmnote{ (#3)}}

\theoremstyle{pedroimage}
\newtheorem{figname}{Figure}

\theoremstyle{pedro}
\newtheorem{lem}{Lemma}[section]

\newtheorem{thm}[lem]{Theorem}

\newtheorem{prop}[lem]{Proposition}

\newtheorem{coro}[lem]{Corollary}

\theoremstyle{remark}

\newtheorem{rmk}[lem]{Remark}

\theoremstyle{pedrodef}

\newtheorem{defn}[lem]{Definition}
\newtheorem{ex}[lem]{Example}

\newcommand{\bs}{\backslash}

\title{Homogeneous coherent configurations from spherical buildings and other edge-coloured graphs}

\author{Pierre Guillot}

\date{}

\numberwithin{equation}{section}

\begin{document}

\maketitle

\begin{abstract}
  {\footnotesize We study a class of edge-coloured graphs, including the chamber systems of buildings and other geometries such as affine planes, from which we build homogeneous coherent configurations (also known as association schemes). The condition we require is that the graph be endowed with a certain distance function, taking its values in the adjacency algebra (itself generated by the adjacency operators). When all the edges are of the same colour, the condition is equivalent to the graph being distance-regular, so our result is a generalization of the classical fact that distance-regular graphs give rise to association schemes.

    The Bose-Mesner algebra of the coherent configuration is then isomorphic to the adjacency algebra of the graph. The latter is more easily computed, and comes with a ``small'' set of generators, so we are able to produce examples of Bose-Mesner algebras with particularly simple presentations.

    When a group~$G$ acts ``strongly transitively'', in a certain sense, on a graph, we show that a distance function as above exists canonically; moreover, when the graph is a building, we show that strong transitivity is equivalent to the usual condition of transitivity on pairs of incident chambers and apartments. As an example, we show that the action of the Mathieu group~$M_{24}$ on its usual geometry is not strongly transitive.

    We study affine planes in detail. These are not buildings, yet the machinery developed allows us to state and prove some results which are directly analogous to classical facts in the theory of projective planes (which {\em are} buildings). In particular, we prove a variant of the Ostrom-Wagner theorem on Desarguesian planes which holds in both the projective and the affine case.
  }

\end{abstract}

\image{.9}{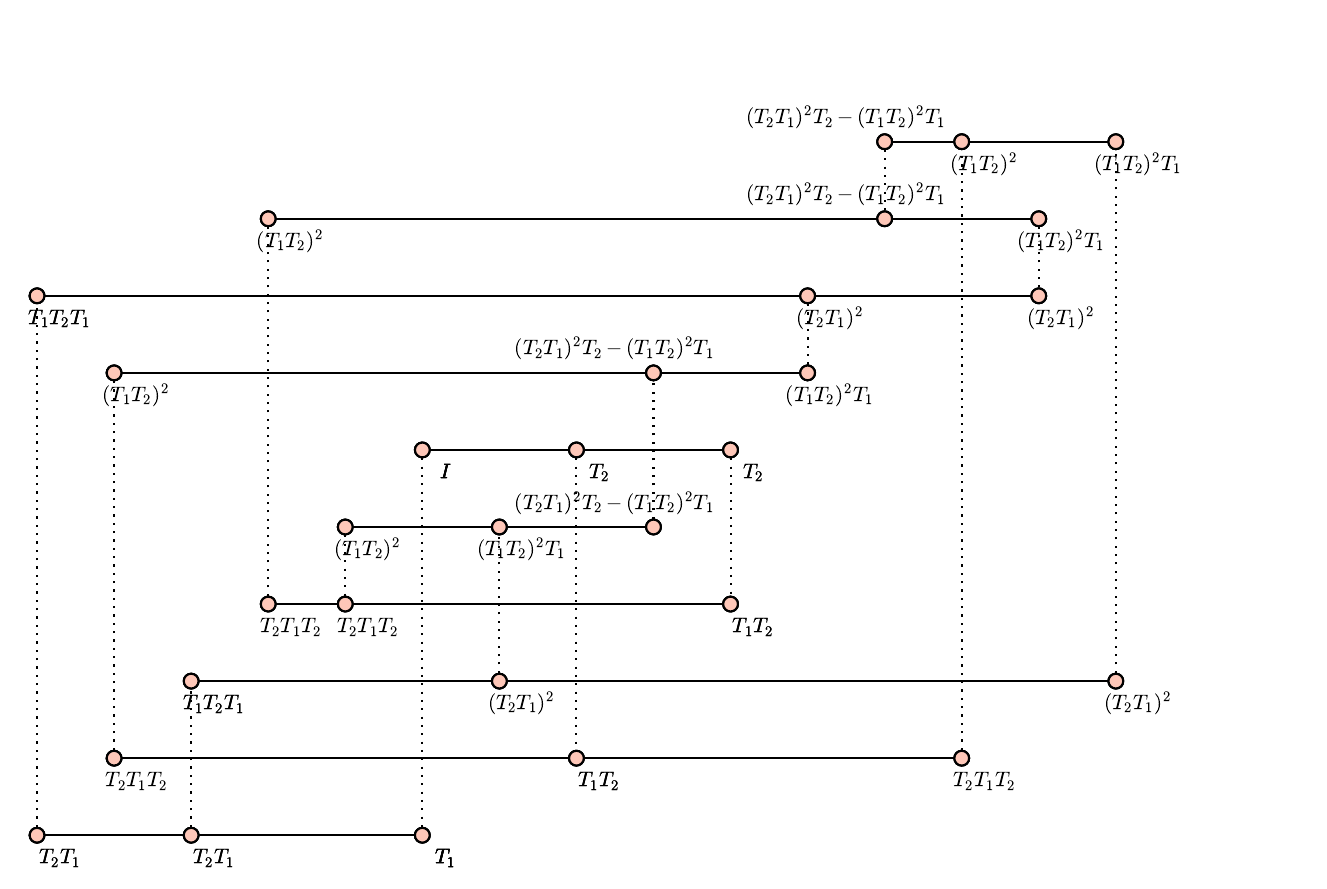}


\section{Introduction}

We start with a rather long Introduction, but many definitions which are given here will not need to be repeated in the sequel. We start by describing graph-theoretical properties of buildings which are then studied for general graphs in the rest of the paper.

We make the blanket assumption that {\em all graphs considered in the paper are finite}. Some of our results hold more generally for locally finite graphs (which may have infinitely many vertices, but with each vertex having finitely many neighbours), and when we feel that this is important enough, we mention it explicitly. This is notably the case in \S\ref{sec-buildings} for the general results on buildings.

\subsection{Buildings as edge-coloured graphs}

In this paper, we view buildings as particular edge-coloured graphs; for many authors, buildings are simplicial complexes with labelled vertices, and the possibility of encoding them as graphs is a result by Tits \cite{local}, see also the equivalence of categories presented in \cite[Theorem 1.3.1]{charlot}. But we follow Weiss \cite{weiss} in taking this view as the definition, on which we expand very much in \S\ref{sec-buildings}. For now, let us describe some special properties of buildings which can be formulated in graph-theoretical language.

We need some general terminology. For any edge-coloured graph~$\Gamma $, we shall systematically write~$V = \C[\Vert(\Gamma )]$ for the complex vector space freely generated by the vertices. If~$\I$ denotes the set of colours, then we can define for~$i \in \I$ the operator~$T_i \colon V \longrightarrow V$ by
\[ T_i (x) = \sum_{y \sim_i x} \, y  \]
where~$y \sim_i x$ means that~$x$ and~$y$ are connected by an edge of colour~$i$.  We call these the {\em adjacency operators}, and the subalgebra~$\A(\Gamma )$ of~$\End(V)$ which they generate is the {\em adjacency algebra} of~$\Gamma $.

A building~$\Gamma $ has an associated Coxeter group~$W$, with distinguished set of generators~$S = \{ s_i \colon i \in \I \}$; these generators are indexed by the set~$\I$, which serves also as the set of colours for the graph. We have enough notation to state our first theorem, which is a re-discovery of a result by Parkinson \cite{parkinson}. Our proof is a little different, so we have decided to include it, see~\cref{rmk-keep-parkinson-quiet}.

\begin{thm}[\cref{thm-adjacency-building-is-hecke} below]
Let~$\Gamma $ be a locally finite building, and suppose that for each~$i \in \I$ there is an integer~$q_i$ such that each vertex of~$\Gamma $ is incident with~$q_i$ edges of colour~$i$. Then~$\A(\Gamma )$ is isomorphic to the Iwahori-Hecke algebra of~$(W,S)$ with parameters~$q_i$. Thus the generators~$T_i$, for~$i \in \I$, satisfy
\[ (T_i - q_iI)(T_i+I) = 0  \]
as well as
\[ T_i T_j T_i \cdots  = T_j T_i T_j \cdots   \]
with~$m_{ij}$ terms in the products on either side, where~$m_{ij}$ is the order of~$s_is_j \in W$. The algebra~$\A(\Gamma )$ has a basis~$\W= \{ T_w : w \in W \}$ indexed by the elements of~$W$.
\end{thm}

\begin{rmk}
This was probably a ``folk'' result for a long time, with many papers using the phrase ``Hecke algebra'' for the adjacency algebra of certain graphs which are not even buildings, for example \cite{ott}. See \cref{rmk-uses-hecke} below for more uses of the phrase ``Hecke algebra'' which may explain further why it has been used for adjacency algebras.
\end{rmk}

Buildings are endowed with a distance function
\[ \delta \colon \Vert(\Gamma ) \times \Vert(\Gamma ) \longrightarrow W  \]
with certain properties which we recall below (indeed, this can be taken as the defining feature of buildings, so it is a deep fact). From the theorem, we see that we can identify~$W$ with the subset~$\W$ of the adjacency algebra, {\em via} $w \mapsto T_w$. Under this identification, we have a map~$\delta \colon \Vert(\Gamma ) \times \Vert(\Gamma ) \longrightarrow \A(\Gamma )$, and it has the following properties :

\begin{enumerate}
\item[(Ar1)] If we fix a vertex~$x$ and consider~$\W = \big\{ \delta (x, y) ~|~ y \in \Vert(\Gamma ) \big\}$, then~$\W$ does not depend on~$x$, and it is a basis for~$\A(\Gamma )$.
\item[(Ar2)] For a given~$T \in \W$, and a vertex~$x$, we have
\[ T(x) = \sum_{\delta (x, y) = T} \, y \, .  \]
\end{enumerate}

We prove this in \cref{coro-buildings-have-architecture}. Of course in (Ar1) the set~$\W$ is the same as above, but we have formulated this property in a way which makes sense for any graph. In fact we make the following definition.

\begin{defn}
Let~$\Gamma $ be an edge-coloured graph. A map~$\delta \colon \Vert(\Gamma ) \times \Vert(\Gamma ) \longrightarrow \A(\Gamma )$ satisfying (Ar1) and (Ar2) will be called an {\em architecture} on~$\Gamma $. We will sometimes call~$\W$ the {\em Coxeter basis} of~$\A(\Gamma )$ (with respect to~$\delta $).

\end{defn}

Motivation for singling out these two axioms will be provided in the next section. For now, we point out that when~$\Gamma $ involves only one colour, then the existence of an architecture is essentially equivalent to~$\Gamma $ being {\em distance-regular} (see \cref{prop-arch-mono-distance-regular} for a precise statement). Also, it may be worth mentioning that when (Ar1) and (Ar2) hold, a number of satisfying other properties follow: thus $\delta (y, x)$ is the transpose of~$\delta (x, y)$, the equality~$\delta (x, y) = I$ is equivalent to~$x=y$, and also~$\delta $ is always~$\Aut(\Gamma )$-invariant, etc. See~\cref{prop-arch-basic-props}.

Another structure on buildings (which can also be taken as the definition) is the existence of {\em apartments}: when~$\Gamma $ is a building with associated Coxeter group~$W$ and set of generators~$S$, the apartments are certain subgraphs of~$\Gamma $ isomorphic to the Cayley graph of~$(W,S)$, satisfying certain axioms. A group~$G$ acting on~$\Gamma $ is traditionally said to act {\em strongly transitively} if it is transitive on pairs~$(x, A)$ where~$x$ is a vertex and~$A$ is an apartment containing~$x$. For our next result, let us focus on finite buildings.

\begin{thm}[\cref{thm-strong-actions-equiv-conditions} below]
  Let~$\Gamma $ be a finite  building, and let~$G$ be a group acting on~$\Gamma $ (by automorphisms of edge-coloured graphs). Then the following statements are equivalent:

  \begin{enumerate}
  \item $G$ acts strongly transitively;
    \item $G$ acts vertex-transitively and $\A(\Gamma ) = \End_G(V)$.
\end{enumerate}

\end{thm}

Here~$\End_G(V)$, the algebra of operators on~$V$ commuting with the action of~$G$, always satisfies~$\A(\Gamma ) \subset \End_G(V)$.

\begin{rmk} \label{rmk-uses-hecke}
In particular, we see from the two theorems that, if~$\Gamma $ is a building with a strongly transitive $G$-action, then~$\End_G(V)$ can be described as a Iwahori-Hecke algebra; in this form, this is certainly a well-known result, see for example \S6.2 in the classic \cite{gar}. In the literature, the term Hecke algebra is sometimes used for algebras of the form~$\End_G(\C[G/B])$ (as in~\cite{lux}, definition after 1.2.16), and sometimes for a slightly more general type of algebras (as in~\cite{CRI}, discussion around definition 11.22).
\end{rmk}

The theorem motivates the following definition.

\begin{defn} \label{defn-strongly-transitive}
  When the action of~$G$ on the edge-coloured graph~$\Gamma $ is vertex-transitive, and satisfies~$\End_G(V) = \A(\Gamma )$, we say that~$G$ acts {\em strongly transitively} on~$\Gamma $, or that its action is strongly transitive. When~$\Gamma $ admits such an action for some group~$G$, or equivalently when this holds for~$G = \Aut(\Gamma )$, we say that~$\Gamma $ is a strongly transitive graph.
\end{defn}

For spherical buildings, the existence of a strongly transitive actions is equivalent to the {\em Moufang} condition, see~\cite{weiss}. So another good name for a strongly transitive graph would be a Moufang graph.

Our next theorem relates this concept to the previous one, and is not restricted to buildings :

\begin{thm}[\cref{thm-strong-trans-implies-arch} and \cref{rmk-two-archs-on-buildings} below] \label{thm-intro-moufang-implies-arch}
Suppose~$\Gamma $ is strongly transitive. Then there exists a unique architecture on~$\Gamma $. In particular, when~$\Gamma $ is also a building with Coxeter group~$W$, this architecture agrees with the one coming from the distance function with values in~$W$.
\end{thm}

We also prove that, when a group acts strongly transitively on a graph, it has a certain ``Steinberg'' representation, which is irreducible (Theorem~\ref{thm-steinberg}). Using this, one can prove that certain actions are {\em not} strongly transitive, and we give an example involving the usual geometry for the Mathieu group~$M_{24}$ (see Example~\ref{ex-M24}).

\subsection{Homogeneous coherent configurations}

Architectures on graphs will allow us to construct coherent configurations. Let us recall the definition.

Let~$m \ge 1$ and~$d \ge 0$ be integers. A {\em homogeneous coherent configuration} of degree~$m$ and rank~$d+1$ is a set of non-zero $m\times m$-matrices with entries in~$\{ 0,1 \}$, which can be enumerated as~$A_0, \ldots, A_d$, such that:

\begin{enumerate}
\item[(a)] $A_0  = I$, the identity matrix,
\item[(b)] $A_0 + \cdots + A_d = J$, the all-one matrix,
\item[(c)] for each index~$i$, we have~$A_i^t \in \{ A_0, \ldots, A_d \}$, where~$A_i^t$ is the transpose of~$A_i$,
  \item[(d)] for all pairs~$i,j$ of indices, the product~$A_iA_j$ is in the linear span of~$A_0, \ldots, A_d$.
\end{enumerate}
By (d), we can write
\[ A_i A_j = \sum_{k=0}^d a_{ijk} A_k \, ,   \]
and by (b) we see that~$a_{ijk} \in \n$ (indeed each~$a_{ijk}$ is a coefficient of~$A_iA_j$). These are called the {\em intersection numbers} of the configuration. (In what follows, we shall frequently say {\em configuration} or {\em coherent configuration} for brevity, but homogeneous coherent configurations are always meant; while more general definitions exist, there are not used in this paper.)

We can think of~$A_i$ as the adjacency matrix of an oriented graph on~$X=\{ 1, \ldots, m \}$, and collectively a coherent configuration can be thought of as an assignment of colours, from the set~$\{ 1, \ldots, d \}$, to the edges of the complete, oriented graph on~$X$ (with no loops). This explains why~$m$ is sometimes called the number of vertices of the configuration; $d$ is the number of colours used, and it is also called the number of classes of the configuration.

Consider the important case when each matrix~$A_i$ is symmetric. Then (d) implies that each~$A_iA_j$ is also symmetric, so that~$A_i A_j = A_jA_i$. In this case, the coherent configuration is called a {\em symmetric association scheme}. In this situation, we can think of~$A_i$ as the adjacency matrix of an unoriented graph (to be called simply a graph in this paper) on~$X$, and a symmetric association scheme is a colouring of the edges of the complete graph on~$X$.

The algebra generated (for us, over~$\C$) by the matrices in a coherent configuration will be called its {\em Bose-Mesner} algebra (a terminology which is standard in the case of symmetric association schemes). It has complex dimension~$d+1$.

Coherent configurations were introduced in 1975 by Higman~\cite{higman}. A key objective, among others which we choose not to discuss here, was to generalize the following basic example.  Suppose the finite group~$G$ acts transitively on a finite set~$X$. For each orbit of~$G$ on~$X\times X$, which is a set of ordered pairs of elements of~$X$, consider the oriented graph on~$X$ whose edges are precisely those pairs. Then collectively, these oriented graphs (or rather, their adjacency matrices) form a coherent configuration, to be called henceforth the {\em basic configuration} obtained from the~$G$-action. (Here we use the maximum number of colours on the edges of the complete, oriented graph on~$X$, in such a way that the action of~$G$ preserves the colours.) The Bose-Mesner algebra for this example is~$\End_G(\C[X])$, the algebra of operators on~$\C[X]$ commuting with the action of~$G$ (see \S\ref{subsec-double-cosets}).

\begin{rmk}
Symmetric association schemes, which may be seen as the ``commutative coherent configurations'', actually arose first, in statistics; their first serious application to algebraic problems was made by Delsarte~\cite{delsarte}. Since then, they have been an essential part of coding theory, for example. See~\cite{godsil} for a modern exposition of the many combinatorial problems related to association schemes. We note that association schemes have also been used in relation with spin networks, see~\cite{spin}. For a textbook on coherent configurations, with a point of view which is somewhat different, see~\cite{mongolito}.
\end{rmk}

In this paper we prove:

\begin{thm}[\cref{coro-archi-gives-config} below] \label{thm-intro-arch-implies-config}
When~$\delta  $ is an architecture on the edge-coloured graph~$\Gamma $, the matrices in~$\W$ form a coherent configuration. The Bose-Mesner algebra of the configuration is just the adjacency algebra~$\A(\Gamma )$.
\end{thm}

If the architecture on~$\Gamma $ comes from a strongly transitive action, then the coherent configuration obtained by the theorem is none other than the ``basic configuration'' just mentioned with~$X= \Vert(\Gamma )$, which is consistent with the Bose-Mesner algebra being~$\End_G(V)$. However, the algebra~$\A(\Gamma )$ is much easier to compute with than~$\End_G(V)$, and comes equipped with a small set of generators. An example will make this clear.

\begin{ex} \label{ex-peter-intro}
Let~$G= S_5$, the symmetric group on~$5$ letters, and~$B= \langle (12), (34) \rangle \cong C_2 \times C_2$. The action of~$G$ on~$G/B$ defines a coherent configuration with~$11$ matrices of size~$30 \times 30$. It is quite possible to write them down (with the help of a computer), but one would like to work with fewer generators, and have a better understanding of the configuration. In \S\ref{sec-compute}, we will see that~$G$ acts strongly transitively on the graph~$\Gamma $ which is displayed on the first page of this paper, and there is a vertex~$x$ with stabilizer~$B$. By~\cref{thm-intro-arch-implies-config}, there is an architecture on~$\Gamma $, and on the picture we have indicated next to each vertex~$y$ the value of~$\delta (y, x)$ (so the vertex~$x$ is the one with the letter~$I$ next to it, standing for the identity matrix). We will see that computing~$\A(\Gamma )$, which is also the Bose-Mesner algebra~$\End_G(V)$ of the configuration, is rather straightforward: it is generated by~$T_1$ and~$T_2$ satisfying~$T_1^2 = I$, $T_2^2  = T_2 + 2I$, and
\[ (T_1 T_2)^3= (T_2T_1)^2(I + T_2) - T_1T_2T_1T_2T_1  \, ,   \]
\[ (T_2T_1)^3 = (I+T_2)(T_1T_2)^2 - T_1T_2T_1T_2T_1   \, .  \]
(The second by transposing the first.) The Coxeter basis is
\[ \W= \big\{
I, \, T_1, \, T_2, \, T_2T_1, \, T_1T_2, \, T_1T_2T_1, \, T_2T_1T_2, \, (T_1T_2)^2, \, (T_2T_1)^2, \, (T_1T_2)^2T_1, \, T_2(T_1T_2)^2 -  (T_1T_2)^2T_1 \big\} \, .  \]
One deduces the intersection numbers easily, and we will explain how this can be used to find with moderate effort all the subgroups which are intermediate between~$B$ and~$G$.
\end{ex}

\topnextpage{
\[ \footnotesize{\left(\begin{array}{rrrrrrrrrrrrrrrrrrrrr}
a & b & b & d & f & f & d & f & f & e & c & e & e & c & e & f & d & f & f & d & f \\
b & a & b & c & e & e & d & f & f & f & d & f & f & d & f & e & c & e & f & d & f \\
b & b & a & d & f & f & c & e & e & f & d & f & f & d & f & f & d & f & e & c & e \\
e & c & e & a & b & b & f & d & f & d & f & f & d & f & f & e & c & e & f & f & d \\
f & d & f & b & a & b & e & c & e & c & e & e & d & f & f & f & d & f & f & f & d \\
f & d & f & b & b & a & f & d & f & d & f & f & c & e & e & f & d & f & e & e & c \\
e & e & c & f & d & f & a & b & b & d & f & f & f & f & d & f & f & d & e & c & e \\
f & f & d & e & c & e & b & a & b & c & e & e & f & f & d & f & f & d & f & d & f \\
f & f & d & f & d & f & b & b & a & d & f & f & e & e & c & e & e & c & f & d & f \\
d & f & f & e & c & e & e & c & e & a & b & b & f & d & f & d & f & f & d & f & f \\
c & e & e & f & d & f & f & d & f & b & a & b & e & c & e & d & f & f & d & f & f \\
d & f & f & f & d & f & f & d & f & b & b & a & f & d & f & c & e & e & c & e & e \\
d & f & f & e & e & c & f & f & d & f & d & f & a & b & b & f & f & d & e & e & c \\
c & e & e & f & f & d & f & f & d & e & c & e & b & a & b & f & f & d & f & f & d \\
d & f & f & f & f & d & e & e & c & f & d & f & b & b & a & e & e & c & f & f & d \\
f & d & f & d & f & f & f & f & d & e & e & c & f & f & d & a & b & b & c & e & e \\
e & c & e & c & e & e & f & f & d & f & f & d & f & f & d & b & a & b & d & f & f \\
f & d & f & d & f & f & e & e & c & f & f & d & e & e & c & b & b & a & d & f & f \\
f & f & d & f & f & d & d & f & f & e & e & c & d & f & f & c & e & e & a & b & b \\
e & e & c & f & f & d & c & e & e & f & f & d & d & f & f & d & f & f & b & a & b \\
f & f & d & e & e & c & d & f & f & f & f & d & c & e & e & d & f & f & b & b & a
\end{array}\right) \, ,}  \]

\begin{figname} \label{fig-big-matrix}
The algebra of matrices of this form, where $a$, $b$, $c$, $d$, $e$, $f \in \C$, is isomorphic to the Iwahori-Hecke algebra of type~$\operatorname{A}_2$ with~$q=2$.
\end{figname}

}

The coherent configurations obtained in this way can surprise us even in relation with buildings and Iwahori-Hecke algebras. As an example, consider the Iwahori-Hecke algebra of type~$\operatorname{A}_2$ with parameter~$q$; in other words, the complex algebra generated by~$T_1$ and~$T_2$ satisfying~$(T_i + 1)(T_i-q) = 0$ for~$i=1, 2$ as well as~$T_1T_2T_1 = T_2T_1T_1$. The existence of the Fano plane (the Desarguesian projective plane of order~$2$, see below)
gives us a building whose adjacency algebra is obtained by specializing~$q$ to~$2$, and the corresponding configuration gives a representation of the latter as the set of all matrices of the form described on \cref{fig-big-matrix}, see next page. This is explained in \S\ref{subsec-projective-planes}.

\subsection{Affines planes}

We come to the class of examples which justifies the machinery presented in this paper. From an affine plane, we will construct an edge-coloured graph~$\Gamma $. The latter is clearly not a building. Yet we can construct an architecture~$\delta $ on~$\Gamma $ with notable ease, and during the construction we see that~$\delta (x,y)$ is a natural ``measure'' of the relative geometric positions of~$x$ and~$y$. The Coxeter basis is introduced even before we can describe the adjacency algebra fully, and the notation and language we use seem very convenient to perform this description.

We need some vocabulary. A {\em line space} is a pair~$(P,L)$ where~$P$ is a set of elements called ``points'', and $L$ is a set of subsets of~$P$ called ``lines'' (of cardinality~$\ge 2$). For example, $P$ and $L$ may be respectively the set of vertices and edges of a simple graph.


A line space $(P,L)$ defines an edge-coloured graph~$\Gamma = \chamber(P,L)$, called its {\em chamber system}, and defined as follows. The vertices are all the pairs~$(p,\ell) \in P \times L$ such that~$p \in \ell$ (these pairs are often called {\em flags}); we place an edge of colour~$1$ between~$(p, \ell)$ and~$(p', \ell)$ when~$p' \ne p$, and we place an edge of colour~$2$ between~$(p, \ell)$ and~$(p, \ell')$ when~$\ell' \ne \ell$. For example, when we perform this construction with the famous Petersen graph, seen as a line space, we obtain the graph drawn on the frontispiece and and already mentioned in~\cref{ex-peter-intro}.

Results such as \cite[Theorem 3.4.6]{buek} show that the chamber system of a line space, or more generally of a geometry, retains a lot of information about it, and indeed the two points of view are almost equivalent.

Consider now {\em linear line spaces}, which are by definition the line spaces with the property that any two points are incident with exactly one common line. A {\em projective plane} is a linear line space with the extra property that any two lines are incident with exactly one common point; an {\em affine plane} is a linear line space with the extra property that, given a line~$\ell$ and a point~$p$ not incident with it, there is exactly one line~$\ell'$ with~$p \in \ell'$ which is parallel to~$\ell$ (that is, there is no point incident with both~$\ell$ and~$\ell'$). Projective and affine spaces are also required to satisfy certain non-triviality conditions which we ignore in this Introduction.

We can use projective planes to shed light on our results about affine planes below. Projective planes correspond exactly, {\em via} the chamber system construction, to the buildings with associated Coxeter group~$S_3$ (= of type $\operatorname{A}_2$), also known as ``generalized triangles''. Each finite projective plane~$\P$ has an order~$q$, such as there are~$q+1$ points incident with each line, and vice-versa. From the results about buildings, we deduce the description of~$\A(\chamber(\P))$, which is the Iwahori-Hecke algebra already mentioned: it is generated by~$T_1$ and~$T_2$ satisfying~$(T_i-qI)(T_i+I) = 0$ and~$T_1T_2T_1 = T_2T_1T_2$. It has dimension~$6$, with Coxeter basis $I, T_1, T_2, T_1T_2, T_2T_1, T_1T_2T_1$.

Here are some of our results on finite affine planes, which are clearly analogous, yet they seem new. Recall that such a plane as an order~$q$, such that each line is incident with~$q$ points, and each point is incident with~$q+1$ lines.

\begin{thm} \label{thm-intro-affine}
  Let~$\AP$ be an affine plane of order~$q$, and let~$\Gamma = \chamber(\AP)$ be its chamber system.
  \begin{enumerate}
\item The incidence algebra of~$\Gamma $, as a complex algebra with two distinguished generators, depends only on~$q$. It is isomorphic to the algebra~$\Aff(q)$ generated by~$T_1$ and~$T_2$ subject to
\[ (T_1 - (q-1)I)(T_1+I) = 0 \, , \quad (T_2 - qI)(T_2 + I) = 0 \, ,  \]
as well as
\[ (T_1T_2)^2 = (q-1)T_2T_1 + (q-1)T_2T_1T_2 - T_1T_2T_1,  \]
and
\[ (T_2T_1)^2 = (q-1)T_1T_2 + (q-1)T_2T_1T_2 - T_1T_2T_1.  \]
It has dimension~$7$, with basis~$I, T_1, T_2, T_1T_2, T_2T_1, T_1T_2T_1, T_2T_1T_2$.

\item $\Gamma $ has a canonical architecture. The associated Coxeter basis (that is, the list of matrices forming a coherent configuration) is $I, T_1, T_2, T_1T_2, T_2T_1, T_1T_2T_1, T_2T_1T_2 - T_1T_2T_1$.


\item The~$\Aff(q)$-module~$V=\C[\Vert(\Gamma )]$ also depends only on~$q$, up to isomorphism.

  \item Let~$\Gamma_0$ be the simple graph obtained from~$\Gamma $ by forgetting the colours of the edges. Then the adjacency eigenvalues of~$\Gamma_0$ are~$-2$ with multiplicity $(q-1)^2(q+1)$, then~$q-2$ with multiplicity~$q$, as well as $2q-1$ with multiplicity~$1$, and finally
\[ \frac{2q - 3 \pm \sqrt{4q+1}}{2} \, ,   \]
each with multiplicity~$q^2 -1$.


\end{enumerate}
\end{thm}

See \cref{prop-adj-for-affine-planes} for (1) and (2), \cref{lem-V-as-adj-module} and \cref{coro-eigenvalues} for (3) and (4).

In (2) the Coxeter basis appears to be always the same, when the elements are expressed in terms of~$T_1$ and~$T_2$, but non-isomorphic affine planes of the same order give rise to non-isomorphic coherent configurations -- or in other words, the matrix representations are different.

We add that in (3), we actually have a complete description of~$V$, and (4) is just an illustration. Of course the result of (4) depends only on~$q$, so that the various affine planes of order~$q$ give rise to a family of isospectral graphs.


The next theorem characterizes the planes whose chamber systems are strongly transitive. Recall that a projective plane is called Desarguesian when it consists of the linear subspaces of dimension~$1$ and~$2$ in~$\f_q^3$, for some prime power~$q$, with incidence defined from inclusion; likewise, an affine plane is called Desarguesian when it consists of the affine subspaces of dimension~$0$ and~$1$ in~$\f_q^2$. (Of course these are the ``obvious'' examples of such planes.)

\begin{thm}[\cref{thm-moufang-affine-desargues}, \cref{thm-moufang-proj-desargues} below] \label{thm-moufang-desargues-intro}
  Let~$\mathbf{X}$ be either a projective plane or an affine plane. The following conditions are equivalent:
  \begin{enumerate}
  \item the chamber system~$\Gamma = \chamber(\mathbf{X})$ is strongly-transitive;
    \item $\mathbf{X}$ is Desarguesian.
\end{enumerate}

\end{thm}

Some comments are in order, for many readers will rightly suspect that the statement about projective planes boils down to something well-known. Indeed, when~$G$ acts on~$\P$, it is easy to show that strong transitivity of the induced action on~$\chamber(\P)$ implies that~$G$ acts 2-transitively on points, so that one can deduce the statement from the celebrated result by Ostrom and Wagner \cite{twotrans}. This is certainly the ``right'' way to go about the projective case, since the argument in {\em loc.\ cit}.\ is elementary.

It is only natural to wonder whether the affine case also reduces to such a classical statement. I am indebted to Bill Kantor for discussing this with me in a series of private communications. First, a simple but nice remark is that strong transitivity is equivalent to transitivity on the set of non-degenerate ordered triangles (this holds for both projective and affine planes; see \cref{lem-strong-2-trans} for more on this). In turn, transitivity on such triangles is equivalent to transitivity on the ``affine quandrangles'' mentioned in Theorem 4 from~\cite{twotrans}, and so this result shows that the affine case was also known (although the proof provided is merely an outline). It is interesting that the literature does not seem to contain another mention of the fact that transitivity on triangles implies, for an affine plane, that it is Desarguesian -- for example it is neither recorded in~\cite{dembowski} nor in~\cite{cofman}.

In this paper, we provide independent arguments, which are not very economical: we rely on the classification of linear spaces, started in \cite{moult} and finished in \cite{saxl}, which itself uses, in certain places, the classification of finite simple groups. Our main point here is, on the one hand, that we can treat projective and affine planes on equal footing, for the two arguments we give run along very similar lines, and differ only in the details. On the other hand, and perhaps more importantly, the concepts introduced in this paper have led us to guess the statement in a very natural way.

More generally, it seems to us that strongly transitive (or Moufang) chamber systems on two colours form a reasonable class of graphs to study in earnest. By the results of the present paper, it contains the finite Moufang polygons classified in~\cite{weisstits}, as well as the chamber systems of finite Desarguesian affine planes; in the last section we shall show that {\em clique planes} provide more examples yet.

\subsection{Relationship with other work}

After this paper was circulated online and submitted for publication, a number of connections with existing work have been brought to my attention by several people, including an anonymous referee, and these deserve to be mentioned.

Theorem~\ref{thm-intro-arch-implies-config}, together with the observation that buildings have an architecture (Corollary~\ref{coro-buildings-have-architecture}), gives a construction of a homogeneous coherent configuration from a building. There are alternative approaches: in~\cite[Theorem E (ii)]{Z12}, a configuration is also obtained from a building. The argument was already given in~\cite{Z11}, and in this paper, the construction is subsequently generalized to ``Moore geometries'', which include our example of the Petersen graph.

These homogeneous coherent configurations obtained from buildings seem to agree with ours, but the technicalities are definitely different. The configurations thus obtained have been termed {\em Coxeter schemes}, and they can be characterized in the language of configurations: see~\cite{Z12} (definition before Theorem 5.1.15) or~\cite{mongolito} (definition after Theorem 12.3.3).

In~\cite{Z10}, among other things, homogeneous coherent configurations are constructed from  certain balanced incomplete block designs, and this includes the case of affine planes. Thus parts of Theorem~\ref{thm-intro-affine} can also be found  there.

There are more traces in the literature of a link between buildings and the Bose-Mesner algebras of the associated configurations, see for example Table 2.1.1 and surrounding discussion in ~\cite{russians}. These often pre-date Parkinson's result cited above, describing the algebra as a Iwahori-Hecke algebra. This is in contrast with our approach, as we give our proof of Parkinson's theorem almost at the same time as we prove the existence of architectures (see the very brief proof of Corollary~\ref{coro-buildings-have-architecture}), and so these topics are, for us, very closely related.

We should also mention that architectures seem to fall under the umbrella of ``generalized distances'' which have received some consideration, for example in~\cite{brouwercohen} or more recently in~\cite{abramenko}.

\subsection{Organization \& Acknowledgements}

We start with a study, in \S\ref{sec-arch-strong-actions}, of architectures and strongly transitive actions on general graphs. We are somewhat digressing in \S\ref{sec-compute} where we give an extended computational example. Then in \S\ref{sec-buildings} we specialize to the case of buildings. Finally in \S\ref{sec-affine} we examine affine planes in detail, and briefly conclude with a similar treatment of clique planes.

\bigskip

I am grateful to Nick Gill, Tao Feng and Sam Mattheus for encouraging words about early versions of this paper.

\subsection{Declarations}

During the submission process for {\em Graphs and combinatorics}, it was specified that this section is mandatory. The information is not relevant, but I have resolved to include it.

\begin{itemize}
\item Funding : none

\item Conflicts of interest/Competing interests : none

\item Availability of data and material : non applicable

\item Code availability : non applicable

\end{itemize}

\section{Architectures \& strongly transitive actions} \label{sec-arch-strong-actions}

\subsection{Edge-coloured graphs and their adjacency algebras} \label{subsec-very-general-indeed}

A {\em graph}~$\Gamma $ is given by a set~$\Vert(\Gamma )$ of vertices, and a set~$\Edge(\Gamma )$ of unordered pairs from~$\Vert(\Gamma )$. (Recall our blanket assumption that all graphs are finite unless specified otherwise.) We say that~$\Gamma $ is {\em edge-coloured} (over~$\I$) when there is a surjective map~$\Edge(\Gamma ) \longrightarrow \I$, where~$\I$ is some finite set, usually taken to be~$\I = \{ 1, \ldots, n \}$ for some integer~$n\ge 1$. Of course, the elements of~$\I$ are called the colours, and an edge mapping to~$i \in \I$ under this map is said to be of colour~$i$, and so on. Generally speaking, we use the standard colloquial terminology of ``neighbours'', edges ``incident'' with a vertex, etc.

A {\em path} of length~$k$ in~$\Gamma $ is a sequence~$\gamma = (x_0, x_1, \ldots, x_k)$ of vertices where~$x_{j-1}$ and~$x_{j}$ are joined by an edge for~$0 < j \le k$. We sometimes say {\em gallery} instead of path, and even {\em chamber} instead of vertex (both habits inherited from the theory of buildings). Now suppose that the edge between~$x_{j-1}$ and~$x_{j}$ has the colour~$i_{j} \in \I$; then the {\em type} of~$\gamma $ is~$(i_1, i_2, \ldots, i_{k})$ (a word in the alphabet~$\I$, if you will). The last vertex visited by a gallery~$\gamma $ will be called its {\em end} and will be denoted by~$e(\gamma )$ (so~$e(\gamma ) = x_k$ in the notation above).

We shall write~$V(\Gamma )= \C[\Vert(\Gamma )]$, and most often we just write~$V$ when~$\Gamma $ is understood. Here we adhere to a convention which will help us in computations: we want the elements of~$\End(V)$ to act {\em on the right} on~$V$. Thus we may think of the elements of~$V$ as {\em row} matrices with columns indexed by the vertices of~$\Gamma $, and of the elements of~$\End(V)$ as matrices with rows and columns both indexed by the vertices, and the action of~$T \in \End(V)$ on~$v \in V$ is given by~$ v \cdot T$, the matrix multiplication.

The algebra~$\End(V)$ of linear endomorphisms of~$V$ has the following distinguished elements~$T_i$ for~$i \in \I$, whose actions on the vertices are given by:
\[ x \cdot T_i = \sum_{y \sim_i x} \, y  \]
where, as in the introduction, we put~$x \sim_i y$ when the vertices~$x$ and~$y$ are joined by an edge of colour~$i$. We call these the {\em adjacency operators}. When we think of~$V$ as the space of complex-valued functions~$f$ on~$\Vert(\Gamma )$ with finite support, we have 
\[ (f \cdot T_i)(x) = \sum_{y \sim_i x} \, f(y) \, .   \]
%

Our convention with right actions is taken in order to obtain the following lemma, which will be of frequent use.

\begin{lem} \label{lem-action-monomial}
Let~$x \in \Vert(\Gamma )$. For any sequence~$i_1, \ldots, i_k$ of colours, we have 
\[ x \cdot T_{i_1}T_{i_2} \cdots T_{i_k}  = \sum_\gamma e(\gamma ) \, ,   \]
where~$\gamma $ runs through the galleries of type~$(i_1, \cdots , i_k)$ starting from~$x$.
\end{lem}

\begin{proof}
Obvious by induction on~$k$. 
\end{proof}

Of course, if we had used actions on the left, then this lemma would have involved an ugly reversal of the sequence of colours. 

We now define~$\A (\Gamma )$, the {\em adjacency algebra} of~$\Gamma $, to be the subalgebra of~$\End(V)$ generated by the~$T_i$'s.  Each~$T_i$ is symmetric, with real entries (indeed, entries in~$\{ 0,1 \}$). It follows that for any~$S \in \A(\Gamma )$, the matrix~$S^t$ is also in~$\A(\Gamma )$, where~$S^t$ denotes the transpose of~$S$. Similarly, the conjugate~$\bar S$ of~$S \in \A(\Gamma )$ is also an element of~$\A(\Gamma )$. 

As a result, we can discover new relations in~$\A(\Gamma )$ by taking the transposition of known ones. Say we had established the formula for~$(T_1T_2)^2$ as in (1) of~\cref{thm-intro-affine}; the formula for~$(T_2T_1)^2$ would follow by taking transposes. 

\begin{prop} \label{prop-adj-algebra-is-semisimple}
The adjacency algebra~$\A(\Gamma )$ is semisimple. Also, the Bose-Mesner algebra of a coherent configuration is semisimple.
\end{prop}

\begin{proof}
  Both statements follow from the fact that any subalgebra~$\A$ of~$\End(V)$ (with~$V$ finite-dimensional) which is stable under transposition and complex conjugation is semisimple. This is a classical fact, and an easy exercise.
%
\end{proof}

\begin{ex} \label{ex-complete-graph}
Suppose~$\Gamma $ is a complete graph on~$q+1$ elements: that is, $\Gamma $ has~$q+1$ vertices, each of them connected to~$q$ neighbours, using only one colour for the edges. The adjacency matrix~$T_1$, in the natural basis, is~$J-I$, where all the coefficients of~$J$ are~$1$'s. Standard linear algebra shows that~$(T_1 - qI)(T_1+I) = 0$, and indeed that~$(X- q)(X+1) \in \C[X]$ is the minimal polynomial for~$T_1$. Thus 
\[ \A(\Gamma ) \cong \frac{\C[X]} {(X-q)(X+1)} \cong \C \times \C \, . 
\]
Moreover, $T_1$ is diagonalisable, with the~$q$-eigenspace having dimension~$1$, spanned by the vector 
\[ \sum_{x \in \Vert(\Gamma )} \, x \, .   \]
The~$(-1)$-eigenspace, consisting of all vectors 
\[ \sum_{x \in \Vert(\Gamma )} \lambda _x x \quad\textnormal{with}~ \sum_{x \in \Vert(\Gamma )} \lambda _x = 0 \, , 
\]
has dimension~$q$. It is instructive to use \cref{lem-action-monomial} to check directly, for a vertex~$x$, the relation 
\[ x \cdot T_1^2 = (q-1) x \cdot T_1 + qx \, .   \]
\end{ex}

The computations made in this example will find an echo throughout the paper, for the edge-coloured graphs which are of interest to us usually have special features. It seems worth giving a definition, for future reference.

\begin{defn} \label{def-chamb-system-regular}
We say that~$\Gamma $ is a {\em chamber system} when, for each colour~$i$, the graph obtained by deleting all the edges whose colour is not~$i$ is a disjoint union of complete graphs. Further, a simple graph is called {\em $d$-regular} when each vertex has exactly~$d$ neighbours; an edge-coloured graph~$\Gamma $ is called {\em regular with orders~$(q_i)_{i \in \I}$} when each vertex has~$q_i$ neighbours at the end of an edge of colour~$i$, for each~$i \in \I$. In this second case, if we forget the colours, we obtain a simple graph which is $\left(\sum_i q_i\right)$-regular.
\end{defn}

Typical graphs in this paper will be chamber systems which are regular with orders~$(q_i)_{i \in \I}$. For these, \cref{ex-complete-graph} shows that 
\[ (T_i - q_iI)(T_i+I) = 0  \]
for~$i \in \I$. This explains why relations of this form are ubiquitous in the sequel.

\subsection{Architectures}

The definition of an architecture on~$\Gamma $ was given in the Introduction, but since we now use actions on the right, let us repeat it here: it is a map~$\delta \colon \Vert(\Gamma ) \times \Vert(\Gamma ) \longrightarrow \A(\Gamma )$ such that :

\begin{enumerate}
\item[(Ar1)] If we fix a vertex~$x$ and consider~$\W = \big\{ \delta (x, y) ~|~ y \in \Vert(\Gamma ) \big\}$, then~$\W$ does not depend on~$x$, and it is a basis for~$\A(\Gamma )$.
\item[(Ar2)] For a given~$T \in \W$, and a vertex~$x$, we have 
\[ x \cdot T = \sum_{\delta (x, y) = T} \, y \, .  \]
\end{enumerate}

We collect the basic consequences of the definition in a proposition.

\begin{prop} \label{prop-arch-basic-props}
  Let~$\delta $ be an architecture on~$\Gamma $.
  \begin{enumerate}
  \item $\Gamma $ is connected.
    \item The~$\A(\Gamma )$-module generated by any vertex~$x$ within~$V(\Gamma )$ is faithful, or in other words, it is the regular representation of~$\A(\Gamma )$. In particular, if~$T \in \A(\Gamma )$ fixes a vertex, then~$T=I$.
  \item For vertices~$x, y$, we have~$\delta (x,y) = I \Longleftrightarrow x= y$. 
\item $\delta$  is~$\Aut(\Gamma )$-invariant : 
\[ \delta (gx, gy) = \delta (x, y) \quad (g\in \Aut(\Gamma ), x,y\in \Vert(\Gamma )) \, .  \]
\item If~$\W = \{ A_0, \ldots, A_d \}$, then 
\[ A_i \cdot A_j = \sum_k a_{ijk} A_k \qquad\textnormal{where}~a_{ijk} \in \n \, . \]
In fact, these integers can be interpreted as follows. Let~$x, z$ be vertices such that~$\delta (x, z) = A_k$. Then the number of vertices~$y$ such that~$\delta (x, y) = A_i$ and~$\delta (y,z) = A_j$ is~$a_{ijk}$ (in particular this number depends only on the colours~$i,j,k$, not on~$x$ or~$z$).

\item If~$x, y$ are vertices, then~$\delta (y, x) = \delta (x, y)^t$, the transpose of the matrix~$\delta (x, y)$. In particular~$\W$ is stable under transposition. It follows that~$\delta $ is symmetric (that is, satisfies~$\delta (x, y) = \delta (y, x)$ for all~$x, y$) if and only if~$\A(\Gamma )$ is commutative, if and only if~$\A(\Gamma )$ is comprised entirely of symmetric matrices.
\end{enumerate}
\end{prop}


\begin{proof}
Let~$x, y$ be vertices of~$\Gamma $, and let~$T= \delta (x,y) \in \A(\Gamma )$. By (Ar2), the vertex~$y$ appears in the expression for~$x \cdot T$. However by \cref{lem-action-monomial}, we see that~$x \cdot T$ can only involve vertices in the connected component of~$\Gamma $ containing~$x$. This gives (1).
  
Let~$T \in \A(\Gamma )$, and write~$\W = \{ A_0, \ldots, A_d \}$. From (Ar1), we can write 
\[ T = \sum_{i} \lambda (A_i) \, A_i  \]
for some uniquely defined scalars~$\lambda (A_i) \in \C$, so that for a vertex~$x$:
\[ x \cdot T = \sum_{i} \lambda (A_i) \, x \cdot A_i = \sum_{i} \sum_{y:\delta (x,y) = A_i} \lambda (A_i) \, y = \sum_{y \in \Vert(\Gamma )} \lambda (\delta (x,y)) \, y \, .   \]
Here we have used (Ar2) for the second equality. Thus the scalars~$\lambda (A_i)$ can be recovered from the vector~$x \cdot T$ (and we do get all of them from (Ar1)), so (2) is clear.

Now continue assuming that~$T=I$. We have 
\[ x = \sum_{y \in \Vert(\Gamma )} \lambda (\delta (x,y)) \, y \, .   \]
From this, we draw~$\lambda (\delta (x, x)) = 1$ and~$\lambda (\delta (x, y)) = 0$ for~$y \ne x$. In particular $I = \delta (x, x)$, and certainly~$\delta (x, y) \ne \delta (x, x)$ if~$y \ne x$. We have (3).

Now apply~$g \in \Aut(\Gamma )$ to the identity in (Ar2), so that for~$T \in \W$ one has : 
\[ g \cdot (x \cdot T) = \sum_{\delta (x, y) = T} \, g \cdot y \, .   \]
However~$g \cdot (x \cdot T) = (g \cdot x)\cdot T$, and another application of (Ar2), this time at the vertex~$g \cdot x$, gives: 
\[ (g \cdot x)\cdot T = \sum_{\delta  (g \cdot x, z) = T} \, z \, .   \]
Comparing the last two expressions, we see that~$\delta (g \cdot x, z) = T$ happens precisely when~$z= g \cdot y$ with~$\delta (x, y) = T$. In particular~$\delta  (g \cdot x, g \cdot y) = \delta  (x, y)$, and we have (4).

As for (5), the proposed identity certainly holds for some complex numbers~$a_{ijk}$, simply because~$\W$ is a basis for~$\A(\Gamma )$, and the point is only to show that these are nonnegative integers. However, this is obvious by (Ar2). We leave the interpretation of the integers~$n_{ijk}$ as an exercise (we will never use the result in the sequel). 

We turn to (6). Let~$\langle -, - \rangle$ denote the inner product on~$V$ for which the basis of vertices is orthonormal. A restatement of (Ar2) is that 
\[ \langle x \cdot T, y \rangle = \left\{ \begin{array}{l}
  1 ~\textnormal{if}~ \delta (x, y) = T \, , \\
  0 ~\textnormal{otherwise} \, , 
\end{array}\right.  \]
when~$T \in \W$, and~$x, y$ are arbitrary vertices. Further, if we pick an operator
\[ T = \sum_{i} \lambda (A_i) A_i \in \A(\Gamma ) \, ,   \]
it follows that~$\langle x\cdot T, y \rangle = \lambda (\delta (x, y))$. Now write 
\[ T^t = \sum_{i} \mu  (A_i) A_i \, ,  \]
and use that 
\[ \langle x \cdot T, y \rangle = \langle x, y \cdot T^t \rangle = \langle y \cdot T^t, x \rangle \, ,   \]
to deduce that 
\[ \lambda (\delta (x, y)) = \mu  (\delta (y, x)) \, .   \]
The fact that~$\delta (x, y)^t = \delta (y,x) $ is a particular case.

It is now clear that when~$\delta $ is symmetric, all the elements of the algebra~$\A(\Gamma )$ are symmetric matrices (from (Ar1) and the identity just established). This implies, for colours~$i, j$, that~$(T_i T_j)^t = T_i T_j = T_j T_i$, so~$\A(\Gamma )$ is commutative. Conversely, suppose that~$\A(\Gamma )$ is commutative: since its generators~$T_i$ are symmetric matrices, it is then readily seen that all the elements of~$\A(\Gamma )$ are symmetric, and of course~$\delta $ is then symmetric. 
\end{proof}

\topnextpage{\imageright{.5}{\begin{figname} \label{fig-no-arch} A graph without architecture.\end{figname}}{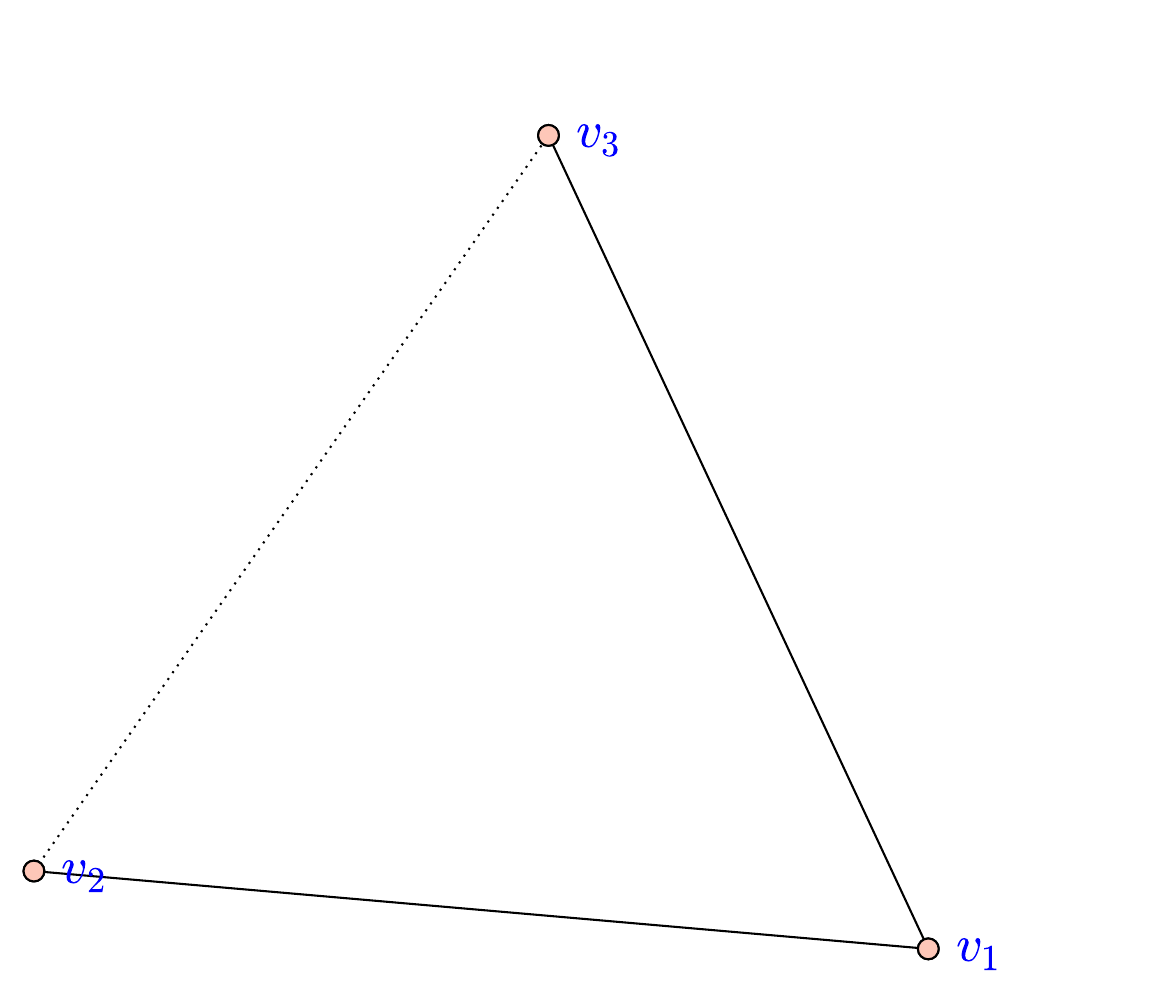}}

\begin{coro} \label{coro-archi-gives-config}
When~$\delta $ is an architecture on~$\Gamma $, the matrices in~$\W$ form a coherent configuration.
\end{coro}

\begin{proof}
From (3) we see that~$I \in \W$, so we may as well assume that the numbering was made so that~$A_0 = I$. From (6) we obtain axiom (c) for coherent configurations, and from (5) we have axiom (d). Property (Ar2) ensures that each matrix~$A_i$ has entries in~$\{ 0,1 \}$, and the fact that any two vertices~$x, y$ have a well-defined distance~$\delta (x,y)$ guarantees that the~$A_i$'s sum up to the all-one matrix, which is axiom (b).
\end{proof}

\begin{ex}
  The simplest example of a graph with architecture is perhaps a complete graph, with only one colour used, where~$\delta (x,x) = I$ for each~$x$, while~$\delta (x,y) = T_1$, the adjacency matrix, when~$x \ne y$.

Consider next the case of a symmetric association scheme, as defined in the Introduction. Recall that this is, on the one hand, a particular case of a coherent configuration; and on the other hand, we may regard such an object as an edge-coloured graph~$\Gamma $, on the set of colours~$\I = \{ 1, 2, \ldots, d \}$, such that the underlying graph (forgetting the colours) is complete. In this situation, put~$\delta (x,y) = T_i$ when~$x \sim_i y$, and~$\delta (x,x) = I$. One checks readily that this is an architecture (generalizing the previous example). The configuration obtained from the corollary is just the one we started with.
  
  Many more examples follow, so here we would like to give examples of graphs for which no architecture can be found. By (1) of the proposition, of course, non-connected graphs are such counter-examples. Consider also the graph on \cref{fig-no-arch}. Here~$v_1$ is connected to~$v_2$ and~$v_3$ with an edge of colour~$1$, while~$v_2$ and~$v_3$ are connected together by an edge of colour~$2$. Suppose~$\delta $ were an architecture on this graph. Then~$\delta (v_1, v_1) = I$ by (3) of the proposition, while~$\delta (v_1, v_2) = \delta (v_1, v_3)$ using (4) (the graph has visibly an automorphism of order~$2$ exchanging~$v_2$ and~$v_3$ and fixing~$v_1$).  We must then have~$v_1 \cdot \delta (v_1, v_2) = v_2 + v_3 = v_1 \cdot T_1$ (the first equality by (Ar2)), so (2) of the proposition shows that~$\delta (v_1, v_2) = T_1$. If (Ar1) were to hold, then~$\{ I, T_1 \}$ would be a basis for~$\A(\Gamma )$; however, $T_2$ is clearly not in the algebra generated by~$T_1$, so~$\delta $ cannot exist.

  In this example, one can check (with more work) that~$\A(\Gamma )$ has dimension 5. This is another good reason why no architecture exists on~$\Gamma $: the dimension of~$\A(\Gamma )$, by (Ar1), cannot be more than the number of vertices.
\end{ex}

\subsection{Distance-regular graphs}

Suppose~$\Gamma $ is a (finite) graph, without colours, which is connected with diameter~$d$. One may form the graph~$\Gamma _i$, for~$0 \le i \le d$, with~$\Vert(\Gamma_i) = \Vert(\Gamma )$ and with an edge between~$x$ and~$y$ in~$\Gamma_i$ if and only if they are at distance~$i$ in~$\Gamma $. Writing~$A_i$ for the adjacency matrix of~$\Gamma_i$, it is natural to ask whether~$\{ A_0, \ldots, A_d \}$ is a symmetric association scheme. Axioms (a), (b) and (c) are visibly satisfied, and one calls~$\Gamma $ {\em distance-regular} when axiom (d) also holds. Very often, this is restated as follows: $\Gamma $ is called distance-regular when there exist integers~$a_{ijk}$, depending only on the indices~$0 \le i, j, k \le d$, such that whenever~$x$ and~$z$ are vertices of~$\Gamma $ at distance~$k$ from one another, the number of vertices~$y$ which are at distance~$i$ from~$x$ and at distance~$j$ from~$z$, is~$a_{ijk}$ (and thus does not depend on~$x$ or~$z$). 

Being distance-regular is closely related to the existence of an architecture on~$\Gamma $. Before we state this, it is best to recall a simple fact about graphs: if the diameter of~$\Gamma $ is~$d$, then the dimension of~$\A(\Gamma )$ is at least~$d+1$ (indeed, let~$A$ be the adjacency matrix of~$\Gamma $; if~$x$ and~$y$ are at distance~$k$ in~$\Gamma $, then the~$(x,y)$ coefficient of~$A^{i}$ is~$0$ when~$i < k$, while the same coefficient is positive for~$A^k$). Note that the dimension of~$\A(\Gamma )$ is the number of distinct eigenvalues for the adjacency matrix. While we may have $\dim \A(\Gamma )> d+1$ in certain cases, it is common to have an equality.

\begin{prop} \label{prop-arch-mono-distance-regular}
  Let~$\Gamma $ be a graph. The following properties are equivalent.

  \begin{enumerate}
  \item $\Gamma $ is distance-regular.
    \item $\Gamma $ has an architecture~$\delta $, and if~$d$ is the diameter of~$\Gamma $, then $\dim \A(\Gamma ) = d+1$.
  \end{enumerate}

Moreover, when these conditions hold, the symmetric association scheme obtained from~$\Gamma $ by the discussion above agrees with that obtained from \cref{coro-archi-gives-config}. In particular, the Bose-Mesner algebra is just~$\A(\Gamma )$.  
\end{prop}

\begin{proof}
  Suppose~$\Gamma $ is distance-regular, and let the matrices~$A_i$ be as above. We put~$\delta (x, y) = A_i$ when~$x$ and~$y$ are at distance~$i$ in~$\Gamma $. For this to define an architecture, the first and main difficulty is to verify that~$A_i \in \A(\Gamma )$; but this is a classical fact, see Corollary 4.1.3 in~\cite{godsil}. We see at the same time that the Bose-Mesner algebra must coincide with~$\A(\Gamma )$ (note that~$A_1$ is the adjacency matrix of~$\Gamma $, which generates~$\A(\Gamma )$). At this point, axioms (Ar1) and (Ar2) are obvious. Also, the dimension of~$\A(\Gamma )$ must~$d+1$, the number of matrices in the association scheme.

  Conversely, suppose that (2) holds. By \cref{coro-archi-gives-config}, we have a coherent configuration~$\W$, which must be a symmetric association scheme (all the matrices in~$\A(\Gamma )$ are symmetric). Since the diameter of~$\Gamma $ is~$d$, we may appeal to Lemma 4.1.2 in~\cite{godsil}, which shows that the matrices~$A_0, \ldots, A_d$ in~$\W$ must be the distance matrices in~$\Gamma $ as above. By definition, $\Gamma $ is then distance-regular.

  The rest is clear.
\end{proof}

Symmetric association schemes obtained in this was are sometimes called {\em metric}, or sometimes $P$-{\em polynomial} (in reference to the fact that all elements in the Bose-Mesner algebra are polynomials in~$A_1$).

\subsection{Double cosets \& the algebra of intertwining operators} \label{subsec-double-cosets}

Here we start with a group~$G$ and a subgroup~$B$. We will recall some uses of the double cosets~$BgB$ for~$g \in G$.

When~$g \in G$, we will write~$\bar g \in G/B$ for its canonical image. The set of orbits of~$B$ in~$G/B$, that is $B \bs (G/B)$, can be identified with the set~$B \bs G /B$ of double cosets, under $B \bar g \longleftrightarrow BgB$. The distinction between~$B \bs (G/B)$ and $B\bs G /B$ is often a pedantic one, but in certain situations it will matter; keep in mind that~$B \bar g$ is a subset of~$G/B$, while~$BgB$ is a subset of~$G$, so the identification is certainly not the identity. The notation~$\bar g$, instead of the popular~$gB$, was also chosen to make the distinction clearer. 

There is also a very well-known identification of~$B\bs G /B$ with the set of orbits of~$G$ on the product $G/B \times G/B$, that is, with~$G \bs \left(G/B \times G/B\right)$: for this, use $BgB \mapsto G(\bar 1, \bar g)$ and $G(\bar h, \bar g) \mapsto B h^{-1}gB$.

Finally, put~$V = \C[G/B]$, the corresponding permutation~$G$-module, and consider the algebra~$\End_G(V)$ of linear maps commuting with the action of~$G$, sometimes called intertwining operators. For~$w \in G$, define $\phi_w \colon V \longrightarrow V$ by the formula
\[ \bar g \cdot \phi_w = \sum_{\bar h \, : \,  h^{-1}g \in BwB} \bar h   \]
for~$\bar g \in G/B$. One checks that~$\phi_w$ is well-defined (that is, the condition~$h^{-1}g \in BwB$ really does depend only on~$\bar h$), that it commutes with the action of~$G$ on~$V$, and also that~$\phi_w$ depends only on~$BwB$.

\begin{rmk} \label{rmk-phi-w-inverse}
  A peculiarity of the notation is that~$\bar 1 \cdot \phi_w$ is the sum of the elements of~$B\overline{w^{-1}}$. Later, we shall work with the operators~$T_w:=\phi_{w^{-1}}$ so as to avoid the inverse.
\end{rmk}

For a nice, elementary proof of the following proposition (which is a classic), see~\cite[Lemma 1.2.15]{lux}:

\begin{prop} \label{prop-lux}
Let~$W \subset G$ be a set of representatives for the double cosets of~$B$ in~$G$. Then the operators~$\phi_w$ for~$w \in W$ form a basis for~$\End_G(V)$.
\end{prop}

In particular, we see that~$\dim \End_G(V) = |B \bs G /B|$ (the two numbers can be simultaneously infinite).

\begin{rmk} \label{rmk-config-on-GB}
We leave it as an exercise to check that, in the canonical basis for~$V$, the operators~$\phi_w$ become precisely the matrices in the ``basic configuration'' associated with the action of~$G$ on~$G/B$, as in the Introduction (use the description of the~$G$-orbits on~$G/B \times G/B$ in terms of double cosets as above).
\end{rmk}

Finally, we describe the correspondence between the~$G$-module structure of~$V$ and its structure as an~$\End_G(V)$-module, assuming that~$G$ is finite now. Write 
\[ V = \bigoplus_{i \in J} m_i S_i  \]
where the~$S_i$'s are simple~$G$-modules, with~$S_i$ not isomorphic to~$S_j$ for~$i \ne j$. Schur's Lemma gives immediately that 
\[ \End_G(V) \cong \prod_{i \in J} M_{m_i}(\C) \, .   \]
However, if we call~$\A$ the algebra on the right hand side, then we know the structure of its simple modules : there is (up to isomorphism) exactly one for each~$i \in J$, afforded by the projection~$\A \longrightarrow M_{m_i}(\C)$. Correspondingly, the isomorphism classes of simple modules of~$\End_G(V)$ are indexed by~$J$, and we pick a representative~$U_i$ for~$i \in J$, noting that~$\dim U_i = m_i$. Thus {\em there is a bijection between the simple~$G$-modules occuring in~$V$, and the simple~$\End_G(V)$-modules} (both sets of isomorphism classes are in bijection with the set~$J$).

There is nice way of expressing this. Rewrite~$m_i S_i$, the direct sum of~$m_i$ copies of~$S_i$, as~$S_i \otimes U_i$, where~$U_i$ is viewed as trivial~$G$-module of dimension~$m_i$. Now, with a little thought, we realize that~$m_i S_i$ (an isotypical summand in~$V$) is stable under the action of~$\End_G(V)$, and indeed that it can be described as a sum of~$\dim S_i$ copies of~$U_i$. It seems reasonable to write~$S_i \otimes U_i$ also for this~$\End_G(V)$-module. (In general, there is no good reason for us to be able to tensor two~$\End_G(V)$-modules together, and~$\End_G(V)$ does not necessarily have a ``trivial'' module, so this is really just suggestive notation.)

In the end, we can summarize the situation by writing 
\[ V = \bigoplus_{i \in J} S_i \otimes U_i \, ,   \]
with the actions and conventions as above. We note that the multiplicity of~$U_i$ as an~$\End_G(V)$-module occuring in~$V$ is~$\dim S_i$, just like the multiplicity of~$S_i$ as~$G$-module is~$\dim U_i$.


\subsection{Strongly transitive actions}

Combining the material above, we let~$\Gamma $ be an edge-coloured graph, and~$G$ be a group acting on~$\Gamma $ by graph automorphisms. For a vertex~$x$, an element~$g \in G$, and an operator~$T \in \A(\Gamma )$, we have thus 
\[ g \cdot (x \cdot T) = (g \cdot x) \cdot T \in V \, .   \]
Equally clear is the inclusion~$\A(\Gamma ) \subset \End_G(V)$.

Now suppose~$G$ acts transitively on~$\Vert(\Gamma )$. If we choose a favourite vertex~$x_0$, and if we let~$B$ denote its stabilizer in~$G$, then we can identify~$\Vert(\Gamma )$ with~$G/B$ whenever convenient. The vector space~$V = V(\Gamma )$ is then seen as~$\C[G/B]$, and the considerations of \S\ref{subsec-double-cosets} apply.

Recall from the Introduction that the action of~$G$ on~$\Gamma $ is called {\em strongly transitive} when it is transitive on~$\Vert(\Gamma )$ and satisfies~$\End_G(V) = \A(\Gamma )$. 

\begin{thm} \label{thm-strong-trans-implies-arch}
  Suppose~$\Gamma $ admits a strongly transitive action. Then~$\Gamma $ has a unique architecture. It is characterized by the property that, given vertices~$x$ and~$y$, the operator~$\delta (x, y)$ is exactly determined by the~$\stab_G(x)$-orbit containing~$y$.

  More precisely, once~$\Vert(\Gamma )$ is identified with~$G/B$ by the choice of a vertex with stabilizer~$B$, we introduce operators~$T_w \in \A(\Gamma )$ in the proof, for each~$w \in G$, where~$T_w$ depends only on~$B w B$;  if~$W \subset G$ is a set of representatives for the double cosets, then~$\W= \{ T_w : w \in W  \}$; and the map~$\delta $ is defined by 
\[ \delta (\bar g, \bar h) = T_w  \]
where~$w$ is such that~$g^{-1} h \in BwB$.
\end{thm}

The uniqueness means, in particular, that choosing another base-point would not alter~$\delta $. However, we caution that the operator~$T_w$ associated to~$w$ does depend on the choice, as is explained at the end of the proof.

\begin{proof}
We choose a base vertex~$x_0$, and for~$g \in G$ put~$\bar g = g \cdot x_0$. This lets us identify~$\Vert(\Gamma )$ with~$G/B$, where~$B$ is the stabilizer of~$x_0$ (with~$\bar g$ identified with the class of~$g$ in~$G/B$, also written~$\bar g$ elsewhere in the paper). At the end of the proof, we study what happens when~$x_0$ is replaced by another vertex. 
  
We have defined in \S\ref{subsec-double-cosets} the operators~$\phi_w \in \End_G(V)$ for~$w \in G$.  By assumption, we have~$\phi_w \in \A(\Gamma )$, and we introduce~$T_w := \phi_{w^{-1}} \in \A(\Gamma )$. Now we put, for vertices~$\bar g, \bar h \in G/B$:
\[ \delta (\bar g, \bar h) = T_w  \]
where~$w$ is such that~$g^{-1} h \in BwB$. This is well-defined, and moreover we note that the condition~$g^{-1}h \in BwB$ is equivalent to~$\overline{g^{-1}y} \in B \bar w$. Let us verify that~(Ar1) and~(Ar2) hold.

First we note an invariance property of~$\delta $. When~$\sigma \in G$, we observe that~$(\sigma g)^{-1}(\sigma h) = g^{-1} h$, implying that~$\delta  (\sigma  \bar g, \sigma  \bar h) = \delta (\bar g, \bar h)$. In other words, $\delta $ is~$G$-invariant. (Of course eventually we shall know that the architecture is~$\Aut(\Gamma )$-invariant, by (4) of \cref{prop-arch-basic-props}.)


The easy part of (Ar1) comes at once: the set~$\W$, seemingly dependent on the choice of a vertex, really depends only on the~$G$-orbit of the vertex; however, the action is assumed to be vertex-transitive, so~$\W$ is independent of all choices.

Let us continue with the vertex~$\bar 1$. For the rest of property (Ar1), we note that~$\delta (\bar 1, \bar h) = T_w$ is equivalent to~$\bar h \in B \bar w$. Thus the set~$\W$ is comprised of all the operators~$T_w$ for~$w \in G$; as observed above, $T_w$ only depends on the double coset~$BwB$ (or equivalently on the orbit~$B \bar w$), and if we pick one~$w$ in each double coset, we obtain a basis for~$\End_G(V) = \A(\Gamma )$ (\cref{prop-lux}). We have (Ar1), and the set~$\W$ is as described in the theorem.

As in \cref{rmk-phi-w-inverse}, we compute for any~$w \in G$:
\[ \bar 1 \cdot T_w = \bar 1 \cdot \phi_{w^{-1}} = \sum_{\bar h : h \in BwB} \bar h = \sum_{\bar h \in B \bar w}  \bar h \, ,  \]
or in other words~$\bar 1 \cdot T_w $ is the sum of the elements in one~$B$-orbit on~$G/B$, namely~$B\bar w$. This is (Ar2) for the vertex~$x= \bar 1$, whence (Ar2) holds in general by $G$-invariance. We have established that~$\delta $ is an architecture on~$\Gamma $.

We have also just seen that~$\delta (\bar 1, \bar h)$ is determined by the~$B$-orbit containing~$\bar h$. By~$G$-invariance, we deduce immediately that~$\delta (x, y)$ is determined exactly by the~$\stab_G(x)$-orbit containing~$y$, for any two vertices~$x, y$. Clearly, this characterizes~$\delta $ among possible architectures. In particular, choosing a different base-point would not have affected~$\delta $.

However, the theorem claims more precisely that~$\delta $ is the unique architecture on~$\Gamma $. Indeed, if~$\delta '$ is another architecture, with Coxeter basis~$\W'$, then consider the spheres 
\[ \{ x \in \Vert(\Gamma ) : \delta' (\bar 1, x) = T \} \, ,   \]
for~$T \in \W'$. Each such sphere is stable under the action of~$B$, by (4) of \cref{prop-arch-basic-props}, or in other words it breaks up as a union of~$B$-orbits. However the number of such spheres is $|\W'| = \dim \A(\Gamma )  = \dim \End_G(V) = | B\bs G /B |$, which is also the number of~$B$-orbits. So each sphere is just one orbit, from which it follows that~$\delta = \delta '$.

A word of caution, to conclude (expanding on the remark before the proof). If we had chosen the vertex~$x_0' = \sigma \cdot x_0$ instead of~$x_0$, where~$\sigma \in G$, then we would have arrived at the same function~$\delta $, as already established (a direct verification is also straightforward). However, this new choice would have defined operators~$T'_w$, for~$w \in G$, and one can check that~$T_w =T'_{\sigma w \sigma^{-1}}$, so here the choice of base-point matters.
\end{proof}

\begin{rmk}
Following up on \cref{rmk-config-on-GB}, we add that the coherent configuration defined from the architecture, which the theorem constructs, is again the ``basic'' configuration. In \S\ref{sec-compute} we will comment on the benefits of the new approach. These are very real, but still, it is important that the concept of architecture does not require a strongly transitive action, and will have other applications, as will be best examplified with the work on affine planes below.
\end{rmk}




\subsection{A representation à la Steinberg}

As an application of our discussion of strongly transitive actions, we discuss an analogue of the {\em Steinberg representation}, which is normally defined for groups acting on buildings only. Then we investigate an example related to the Mathieu group~$M_{24}$.

The next result is the only one in this paper which is more easily stated with simplicial complexes rather than edge-coloured graphs. As claimed in the in Introduction, the two points of view are essentially equivalent, but for simplicity we only recall how to construct a graph from such a complex. We consider {\em labelled} simplicial complexes : the vertices of such a complex are coloured, or in other words there is a map from the set of vertices to~$\I$, and we require that the vertices belonging to a given simplex be of different colours. From a labelled simplicial complex, we construct a graph~$\Gamma $, whose vertices are the maximal simplices of~$X$, and with an edge of colour~$i$ between~$u$ and~$v$ if there is a vertex~$u_i$ of colour~$i$ in~$u$, and a vertex~$v_i$ of colour~$i$ in~$v$, such that $u - \{ u_i \} = v - \{ v_i \}$ (here we assume~$u \ne v$ of course). Note that~$\Gamma $ is then a chamber system (as in Definition~\ref{def-chamb-system-regular}). (Under mild assumptions, one can reconstruct~$X$ from~$\Gamma $, but we skip this discussion; for this, see Theorem~1.3.1 in~\cite{charlot}.)

\begin{thm} \label{thm-steinberg}
Suppose that~$\Gamma $ is obtained from the labelled simplicial complex~$X$, with set of colours~$\I= \{ 1, \ldots, n \}$, and assume that all maximal simplices of~$X$ are of dimension~$n-1$. Assume further that the group~$G$ acts on~$X$, and that the induced action on~$\Gamma $ is strongly transitive. Then the~$G$-module~$H_{n-1}(X, \C)$ is irreducible (or zero). Moreover, it occurs with multiplicity one in~$V= \C[\Vert(\Gamma )]$.
\end{thm}

\begin{proof}
  We assume that~$H_{n-1}(X, \C)$ is nonzero. The complex~$X$ is~$(n-1)$-dimensional, so $H_{n-1}(X, \C) = Z_{n-1}(X)$, the subspace of cycles in degree~$n-1$, or in other words the kernel of the differential $\partial \colon C_{n-1}(X) \longrightarrow C_{n-2}(X)$, where~$C_*(X)$ stands for the chains on~$X$ (with complex coefficients) in degree~$*$.

  Notice that~$C_{n-1}(X)$ can be indentified with~$V= \C[\Vert(\Gamma )]$. Let us choose, for each simplex of~$X$, the orientation given by the natural order on the set~$\I$ of colours. Pick~$f \in V$, and think of~$f$ as a complex-valued function on the set of oriented $(n-1)$-simplices. Then~$\partial(f) = 0$ happens precisely when, for each~$(n-2)$-simplex~$\tau $, we have 
\[ \sum_{\sigma : \tau ~\textnormal{is a face of}~ \sigma  } f(\sigma ) = 0 \, .   \]
(Our choice of orientations has eliminated all possible signs.) Thinking now of~$f$ as a function on the vertices of~$\Gamma $, we see that~$\partial(f) = 0$ happens precisely when~$f$ sums to~$0$ on each 1-residue, that is, on each connected component of the graph obtained from~$\Gamma $ by keeping only the edges of colour~$i$, for each~$i$. (These connected components are in bijection with the~$(n-2)$-simplices of~$X$, which are naturally coloured by~$\I$.) We see finally that~$\partial(f) = 0 \Longleftrightarrow f \cdot T_i = -f$ for each colour~$i$.

Now, assume that~$\A(\Gamma ) = \End_G(V)$. Then the algebra~$\End_G(V)$ is generated by the operators~$T_i$, and it follows that each nonzero~$f \in H_{n-1}(X, \C)$ spans a 1-dimensional~$\End_G(V)$-module. More precisely, as there can be at most one~$\End_G(V)$-module, say~$M$, which has dimension~$1$ and in which all the generators~$T_i$  act as~$-Id$, we see that~$H_{n-1}(X, \C)$ is an isotypical summand in~$V$, splitting as a sum of copies of~$M$.

Our discussion above shows that, as a~$G$-module, the subspace~$H_{n-1}(X, \C)$ is irreducible and occurs with multiplicity~$1$ (since it is an isotypical summand with multiplicity~$\dim M = 1$).
\end{proof}

We can use this as a criterion to prove that a certain action is {\em not} strongly transitive.

\begin{ex} \label{ex-M24}
  The sporadic groups act on various graphs, which are meant to be analogues of buildings. It is natural to investigate the extent to which the analogy actually works, and now we can ask whether the action of the group under scrutiny is strongly transitive. When dealing with buildings, this is a basic requirement.

  Let us focus here on the Mathieu group~$M_{24}$. It acts on a certain {\em geometry} of rank~$3$, from which one defines at once a labelled simplicial complex~$X$ of dimension~$2$ and an edge-coloured graph~$\Gamma $, related as above. This is described at length in Chapter 7 of~\cite{ash2}, of which we shall extract very little indeed. Let us now prove that the action of~$M_{24}$ on~$\Gamma $ is not strongly transitive. If it were, then by the theorem the represention on~$H_2(X, \C)$ would be irreducible. However, $X$ has 6325 vertices, 64515 edges, and 79695 triangles, so its Euler characteristic is 21505. Using that~$X$ is connected, we deduce that 
\[ \dim H_2(X, \C) = 21504 + \dim H_1(X, \C) \ge 21504 \, .   \]
A glance at the character table for~$M_{24}$ reveals that the Mathieu group does not have an irreducible representation of such a large dimension. This contradiction shows that the action is not strongly transitive.
\end{ex}

\subsection{Graphs with the same intersection numbers}

We conclude the generalities with an application of some results of Higman on coherent configurations.  Consider all the graphs having the same adjacency algebra. What do they have in common? Here we must distinguish between several variants of the question. If we merely mean that we have~$\Gamma_1$ and~$\Gamma_2$, two edge-coloured graphs, and an isomorphism of algebras~$\A(\Gamma_1) \cong \A(\Gamma_2)$, then there seems to be very little to say. If we have, on the other hand, an isomorphism of algebras-with-distinguished-generators, so that the elements~$T_1, T_2, \ldots $ correspond to each other (and the same number of colours is employed in both graphs), and if moreover the modules~$V(\Gamma_1)$ and~$V(\Gamma_2)$ are isomorphic as modules over~$\A(\Gamma_1) = \A(\Gamma_2)$, then we may gather some information. For example, for graphs with only one colour, this is the situation when~$\Gamma_1$ and~$\Gamma_2$ have the same eigenvalues (with the same multiplicities), and there is a vast literature on the subject (a recent textbook is~\cite{nica}).

Is it useful to consider graphs with ``the same'' architecture, in relation to the above questions? Again, we must make a distinction, and the danger here is to require too much: if~$\Gamma_1$ and~$\Gamma_2$ define the very same coherent configuration~$\W$, via an architecture, then in practice the adjacency operators~$T_i$ will be in~$\W$, and more often than not we will be able to conclude that~$\Gamma_1$ and~$\Gamma_2$ are isomorphic, which is not the most interesting situation. However, let us suppose that there is again an isomorphism of algebras-with-distinguished-generators between~$\A(\Gamma_1)$ and~$\A(\Gamma_2)$, that~$\Gamma_1$ and~$\Gamma_2$ both have an architecture leading to coherent configurations~$\W_1$ and~$\W_2$, and that the isomorphism~$\A(\Gamma_1) \cong \A(\Gamma_2)$ identifies the basis~$\W_1$ with~$\W_2$. We shall summarize this setup by saying that~$\Gamma_1$ and~$\Gamma_2$ have {\em the same intersection numbers} (for the configurations~$\W_1$ and~$\W_2$ will indeed have the same intersection numbers, cf.\ the Introduction). Then we have:

\begin{prop}\label{prop-criterion-same-V}
  Suppose~$\Gamma_1$ and~$\Gamma_2$ have the same intersection numbers. Then~$V(\Gamma_1)$ and~$V(\Gamma_2)$ are isomorphic as~$\A(\Gamma_1) = \A(\Gamma_2)$-modules. In particular, $\Gamma_1$ and~$\Gamma_2$ have the same number of vertices.
\end{prop}

\begin{proof}
This follows from a nice result by Higman~\cite[\S5]{higman}, which states that the intersection numbers determine all the other ``parameters'' of the configuration, including the ``irreducible degrees'' and the ``multiplicities'', which together describe~$V(\Gamma_i)$.
\end{proof}

More generally, the result by Higman quoted in the proof shows that graphs with the same intersection numbers have a lot of properties in common, and the proposition is just an illustration.

\section{Cell mutliplications \& Computational considerations}\label{sec-compute}

In this section we expand on \cref{ex-peter-intro} from the Introduction. Suppose a finite group~$G$ acts transitively on a set~$X$, so that we can identify~$X$ with~$G/B$ for some subgroup~$B$. Then we know from the opening paragraphs how to define a coherent configuration from this. It is possible to write down explicitly the corresponding matrices. However, when we want to work with the Bose-Mesner algebra, all we know {\em a priori} is that the dimension is~$|B \bs G / B|$ (\cref{rmk-config-on-GB}), and it is unclear whether the algebra can be generated by fewer generators.

Suppose now that~$G$ acts on an edge-coloured graph~$\Gamma $, that the stabilizer of some vertex is~$B$, and that the action is strongly transitive (which may be checked by a dimension count). Then we know that the Bose-Mesner algebra can be generated by as many matrices as there are colours, and the relations between these can be worked out easily.

A possibility that is opened up is the investigation of the ``cell multiplication rules'', as explained below -- roughly this means a description of the groups which are intermediate between~$B$ and~$G$. In the case of buildings, as is well-known, there is a very nice, conceptual description of this (see for example Proposition 11.16 from~\cite{weiss}). In the general case, we must compute, and we want to argue that this is quite feasible in our setup.

We first describe the general problem, and then return to  \cref{ex-peter-intro}. Note that the material in this section is independent from the sequel, and may be skipped.

\subsection{Cell multiplication rules}

Let~$B$ be a subgroup of the group~$G$. Any subset of~$G$ which is stable under multiplication by~$B$ on either side must be a union of double cosets. This applies in particular to a product $BgB \cdot BhB$. We say that we have given ``cell multiplication rules'' when we have offered a recipe for computing the decomposition of any such product explicitly as a union of double cosets. (This is standard terminology in the literature on buildings.) Here we shall do just this under the assumption that~$G$ act strongly transitively on an edge-coloured graph~$\Gamma $, in such a way that~$B$ is the stabilizer of some vertex~$x_0$.  We keep this hypothesis for this section, and we use freely the canonical architecture on~$\Gamma $, as well as the identification of~$\Vert(\Gamma )$ with~$G/B$.

A definition will be useful. We may see a vector~$f \in V(\Gamma )$ as a function $f \colon \Vert(\Gamma )\longrightarrow \C$, so that 
\[ f= \sum_{x \in \Vert(\Gamma )} f(x) x \, .   \]
The {\em support} of~$f$ is then $\supp(f) = \{ x \in \Vert(\Gamma ) : f(x) \ne 0 \}$.

We can now state:

\begin{lem}
Let~$w, v \in G$, let~$T_w, T_v \in \W$ be the corresponding operators, and let~$\bar g,\bar h  \in G/B$ be vertices. Then
\[ \bar h \in \supp(\bar g \cdot T_w T_v) \Longleftrightarrow g^{-1}h \in  BwB \cdot BvB \, .   \]
\end{lem}


\begin{proof}
  We compute:
  \begin{align*}
  (\bar g \cdot T_w) \cdot T_v  & = \left( \sum_{\bar k : g^{-1}k \in BwB} \bar k\right) \cdot T_v\\
    &= \sum_{\bar k : g^{-1}k \in BwB} \quad \sum_{\bar h : k^{-1}h \in BvB} \bar h \, . 
\end{align*}
So~$\bar h \in \supp(T_w T_v(\bar g))$ if and only if we can find~$k \in G$ with~$g^{-1}k \in BwB$ and~$k^{-1}h \in BvB$. When this is the case, we  multiply out and find~$g^{-1} h \in BwB \cdot BvB$. Conversely, if $g^{-1}h = ab$ with~$a \in BwB$ and~$b \in BvB$, then put~$k= h b^{-1}$, so that~$g^{-1} k = a \in BwB$ and~$k^{-1}h = b \in BvB$.
\end{proof}

To formulate this as a ``cell multiplication rule'', select a set~$W= \{ w_1, \ldots, w_d \}$ of representatives for the double cosets, and for a pair of indices~$i,j$, put 
\[ K_{ij} = \{ k : a_{ijk} \ne 0 \} \, , 
\]
where the integers~$a_{ijk}$ are defined by 
\[ T_{w_i} T_{w_j} = \sum_k a_{ijk} T_{w_k} \, .   \]

\begin{coro}
For any~$i,j$, we have 
\[ Bw_iB \cdot Bw_jB = \bigcup_{k \in K_{ij}} Bw_kB \, .   \]
\end{coro}

\begin{proof}
Examine the vector 
\[ v = \bar 1 \cdot T_{w_i} T_{w_j} = \sum_k a_{ijk} \bar 1 \cdot T_{w_k} \, .   \]
Reasoning with the right hand side first, we recall that~$\bar 1 \cdot T_{w_k}$ is the sum of the vertices in~$B\bar w_k$, so~$\supp(v) = \bigcup_{k \in K_{ij}} B \bar w_k$. However, from the lemma we know that~$\bar h \in \supp(v)$ if and only if $h \in Bw_i B \cdot B w_j B$. The corollary follows.
\end{proof}

Here is another way of stating the result. Put 
\[ \A^{01}(\Gamma ) = \left\{ \sum_i n_i T_{w_i} : n_i \in \{ 0,1 \} \right\} \, . 
\]
Define an operation $\odot$ on~$\A^{01}(\Gamma )$ by 
\[ T_{w_i} \odot T_{w_j} = \sum_k \min(1, n_{ijk}) T_{w_k} \, .   \]
Likewise, define 
\[ \sum_i n_i T_{w_i}  \oplus \sum_{i} m_i T_{w_i} = \sum_i \min(1, n_i + m_i) T_{w_i} \, .   \]
The notation~$\A^{01}(\Gamma )$ hides the dependence on the basis~$\W$, even though we are really describing an operation on a coherent configuration. The elements of~$\A^{01}(\Gamma )$ can be identified with the subsets of~$\W$, with~$\oplus$ corresponding to the union, and the operation~$\odot$ is what Zieschang calls the ``complex multiplication'' on a coherent configuration, see~\cite{mongolito}.

It may be worth pointing out, on the other hand, that the definition of~$\A^{01}(\Gamma )$ and its two operations does not depend on the choice of a set~$W$ of representatives, although we have used such a choice for notational convenience.  

We will compare~$\A^{01}(\Gamma )$ with~$\B(G,B)$, which we define to be the set of all subsets of~$G$ which are stable under multiplication by elements of~$B$ on either side, or equivalently, the unions of double cosets of~$B$ in~$G$. Endowed with union and intersection, $\B(G,B)$ is boolean algebra. It also carries a multiplication, unsurprisingly defined by 
\[ X \cdot Y = \{ xy : x \in X, y \in Y \}
\]
for~$X, Y \in \B(G,B)$. This product is distributive with respect to~$\bigcup$, but is not always commutative.

The following is a summary of the discussion; the details should be obvious now.

\begin{prop}
There is a bijection between~$\B(G,B)$ and~$\A^{01}(\Gamma )$, which takes~$BgB$ to~$T_g$ for any~$g \in G$, and under which the operations of union and multiplication on~$\B(G, B)$ correspond respectively to~$\oplus$ and~$\odot$ on~$\A^{01}(\Gamma )$. 
\end{prop}


Assuming as we do that~$G$ is finite, any nonempty subset of~$G$ which is stable under multiplication is a subgroup. In this situation, a subgroup of~$G$ containing~$B$ is just a nonempty~$X \in \B(G, B)$ such that~$X \cdot X = X$. Hence we may state:

\begin{coro}
The subgroups of~$G$ containing~$B$ are in bijection with the nonzero~$X \in \A^{01}(\Gamma )$ such that~$X \odot X = X$.
\end{coro}

For example, the identity corresponds to~$B$, and the sum of all the elements in~$\W$ corresponds to~$G$. Let us now turn to a more involved example.

\subsection{The Petersen graph}

\topnextpage{\imageright{.5}{\begin{figname}Petersen's graph.\end{figname}}{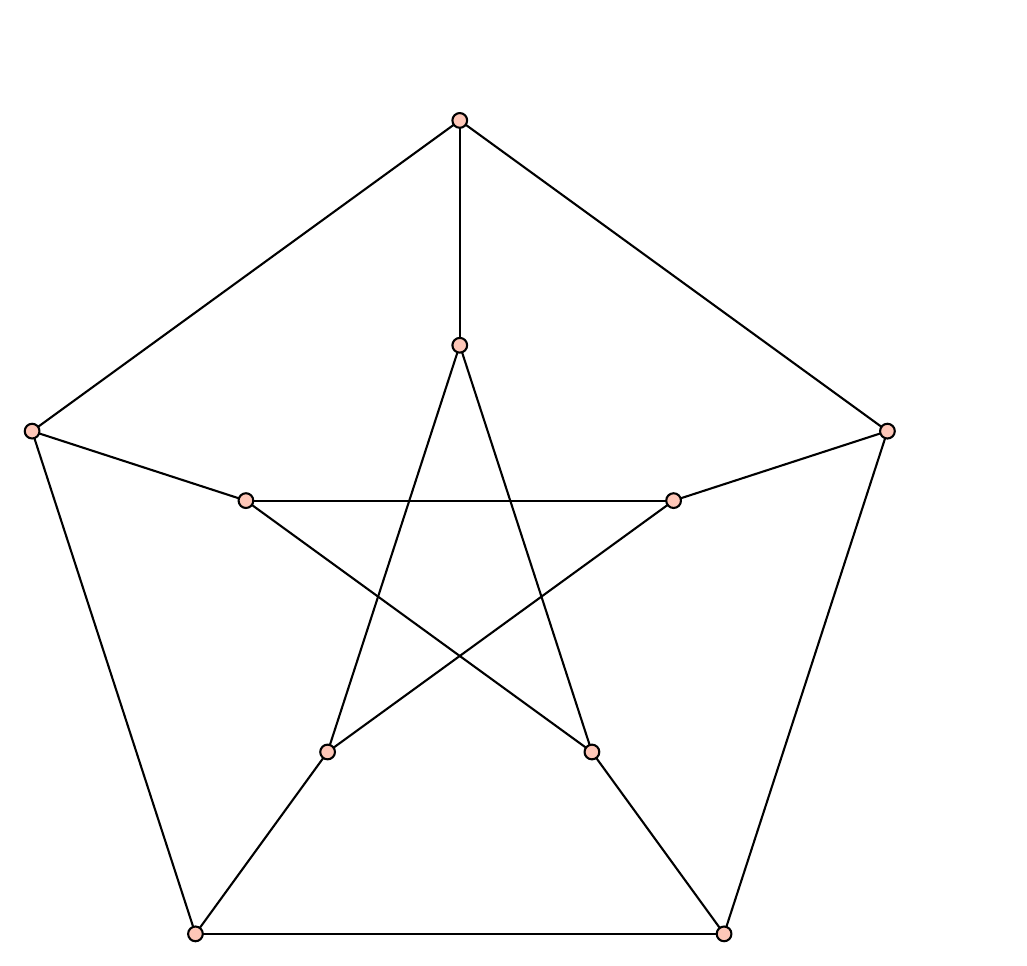}}

We develop \cref{ex-peter-intro}. Much of the heavy lifting was done by a computer, so you should not expect the details of intermediate calculations.

So we consider the edge-coloured graph~$\Gamma $ obtained from the Petersen graph~$\Pi$, itself drawn on the next page, with the general procedure given  in the Introduction for line spaces. The group~$G= S_5$ acts on~$\Pi$, which can be seen most clearly by noting that~$\Pi$ is isomorphic to the graph whose vertices are the unordered pairs from~$\{ 1, 2, 3, 4, 5 \}$, and whose edges are placed between disjoint pairs. In fact~$G$ is the automorphism group of~$\Pi$, although we will not use this. It follows that~$G$ acts on~$\Gamma $, and again it turns out that~$G= \Aut(\Gamma )$.

Once the vertices of~$\Pi$ have been numbered from~$1$ to~$10$, we can label the vertices of~$\Gamma $ with ordered pairs~$(i,j)$ with~$1 \le i, j \le 10$. We pick~$(1,2)$ as our base point, and we let~$B$ be the stabilizer of~$(1,2)$ in~$G$. In the sequel, $(1,2)$ plays the rôle of the vertex named~$\bar 1$ elsewhere in the paper.

We can then ask GAP to compute the double cosets of~$B$ in~$G$. There are~$11$ of them, and GAP even provides representatives~$w_0, \ldots, w_{10}$ (we will not display them here). Of course they appear in a random order, and we will pretend to be lucky later when the elements~$T_{w_i}$ will come out in exactly the most convenient order, when this was really done in hindsight. Write~$W= \{ w_0, \ldots, w_{10} \}$.

Next we ask the computer to determine the adjacency matrices~$T_1$ and~$T_2$. The algebra which they generate, we learn, has dimension~$11$. Since~$\dim \End_G(V) = |B \bs G /B| = 11 = \dim \A(\Gamma )$, we deduce that~$\End_G(V) = \A(\Gamma )$, or in other words, the action is strongly transitive. (Vertex-transitivity is obvious.)

We can work out a presentation for~$\A(\Gamma )$. The relations 
\[ (T_1 - I)(T_1+I) = 0 \, , \quad (T_2 - 2I)(T_2+I) = 0  \]
are expected from \cref{ex-complete-graph}. Having guessed what the standard basis (extracted from the family of monomials~$T_1T_2T_1T_2 \cdots $, not involving any squares) should be, we ask the computer for confirmation, and we learn that the matrices
\[  I, \, T_1, \, T_2, \, T_2T_1, \, T_1T_2, \, T_1T_2T_1, \, T_2T_1T_2, \, (T_1T_2)^2, \, (T_2T_1)^2, \, (T_1T_2)^2T_1, \, T_2(T_1T_2)^2   \]
are linearly independent, so they form indeed our basis. Then we make the computer express~$(T_1T_2)^3$ in this basis, discovering that
\[ (T_1 T_2)^3= (T_2T_1)^2(I + T_2) - T_1T_2T_1T_2T_1  \, .   \]
Taking transposes, we deduce that
\[ (T_2T_1)^3 = (I+T_2)(T_1T_2)^2 - T_1T_2T_1T_2T_1   \, .  \]

Now, we see that the displayed relations form a presentation for~$\A(\Gamma )$. Indeed, any algebra generated by~$T_1$ and~$T_2$ satisfying these relations must have dimension~$\le 11$, since the monomials above are a generating family. Having found one example of algebra of dimension exactly~$11$ where the relations hold, we see that it must be isomorphic to the universal algebra defined thus by generators and relations.

Using this, we can work out the~$1$-dimensional representations of~$\A(\Gamma )$. Under a homomorphism~$\A(\Gamma ) \longrightarrow \C$, the generator~$T_1$ must be sent to~$\pm 1$, and~$T_2$ must be sent to either~$2$ or~$-1$. In each of the four cases, we only have to check whether the remaing relations hold. We find that~$(-1, 2)$ is an impossible combination, but the other three lead to well-defined representations. The one corresponding to the choice~$(1,2)$ is the representation~$C$ mentioned in the proof of \cref{prop-criterion-same-V} below.

Since~$\A(\Gamma )$ is semisimple (\cref{prop-adj-algebra-is-semisimple}), and so must be isomorphic to a product of matrix algebras, we see by trying to write~$11 = 1+1+1+$ a sum of squares of integers~$>1$ that 
\[ \A(\Gamma ) \cong M_2(\C) \times M_2(\C) \times \C \times \C \times \C \, .   \]
The considerations of \S\ref{subsec-double-cosets} apply, with~$\A(\Gamma )$ rather than~$\End_G(V)$, and we discover that~$V(\Gamma )$, as a~$G$-module, involves five different irreducible representations, two of them with multiplicity~$2$, and the remaining three with multiplicity~$1$.

\topnextpage{\imageright{1}{\begin{figname} \label{fig-astuce}
Optimizing the search for intermediate subgroups.
\end{figname}}{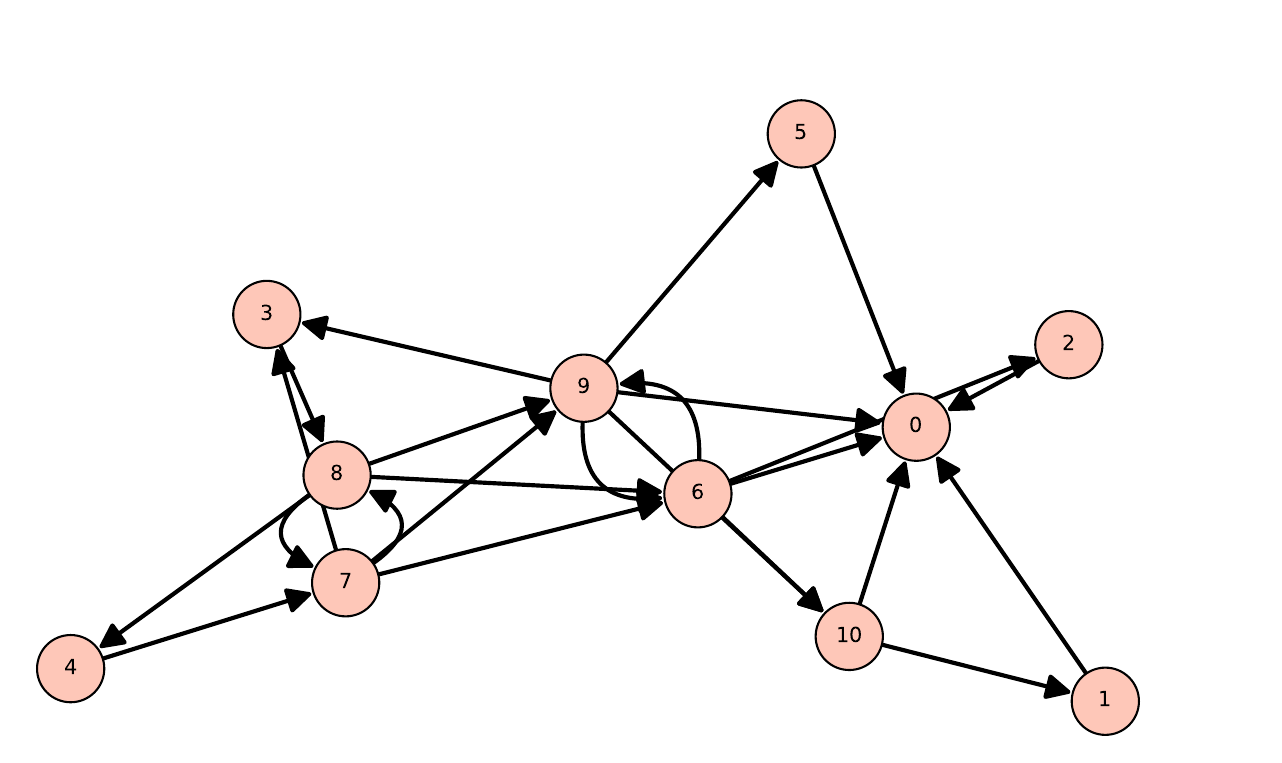}}

We need to compute the Coxeter basis. We know that we do have an architecture, so from (2) of \cref{prop-arch-basic-props}, it is enough to find, for each~$w \in W$, a matrix~$T_w \in \A(\Gamma )$ such that 
\[ \bar 1 \cdot T_w = \sum_{\bar y : y \in BwB} \bar y \, .   \]
We can compute the right hand side with ease at this point. On the other hand, the fact that~$T_w \in \A(\Gamma )$ will be expressed by writing this matrix as a linear combination of the 11 matrices in the standard basis. Assuming we number the vertices starting from~$\bar 1$, the vector~$\bar 1 \cdot T_w$ is the first row of~$T_w$. We solve a linear system, which the theory predicts has a unique solution, and we are done. In the end we find: 
\[ T_{w_0} = I, \, T_{w_1} =T_1, \, T_{w_2} =T_2, \, T_{w_3} =T_2T_1, \, T_{w_4} =T_1T_2, T_{w_5} =\, T_1T_2T_1, \, T_{w_6} =T_2T_1T_2, \, T_{w_7} =(T_1T_2)^2, \,   \]
\[ T_{w_8} =(T_2T_1)^2, \quad \, T_{w_9} =(T_1T_2)^2T_1, \, \quad T_{w_{10}} =T_2(T_1T_2)^2 -  (T_1T_2)^2T_1 \, .   \]
%

The values of~$\delta $ are indicated on the front page. The vertex~$x$ bearing the label~$I$ has been selected, and then each vertex~$y$ bears the label~$\delta (y, x)$. For example, there are four vertices with label~$T_2 T_1 T_2$. These are precisely the vertices at the end of a gallery of type~$(2,1,2)$ starting from~$x$, and they form a~$B$-orbit. Similarly for the other labels. Vertices with the label $T_2(T_1T_2)^2 -  (T_1T_2)^2T_1$ are at the end of a gallery of type $(2,1,2,1,2)$ from~$x$, but they are not at the end of a gallery of type $(1,2,1,2,1)$.

Let us determine all the subgroups of~$G$ containing~$B$. We start with a brute force approach, which is enough to give a complete answer in a matter of seconds. We compute once and for all the various products~$T_{w_i} \odot T_{w_j}$ and store the results, so for example from 
\[ T_{w_8} T_{w_9} = 4 T_{w_1} + 2 T_{w_3} + T_{w_7} \, ,   \]
we have 
\[ T_{w_8} \odot T_{w_9} =  T_{w_1} +  T_{w_3} + T_{w_7} \, .   \]
Then we go through the~$2^{11} -1= 2047$ non-zero elements~$X \in \A^{01}(\Gamma )$, and check whether $X \odot X = X$. We find exactly 6 such elements, so there are 6 groups between~$B$ and~$G$. For instance, $T_{w_0} + T_{w_5}$ is one such element, and the corresponding group is thus $B \cup Bw_5B = \langle B,w_5 \rangle$. The complete list is: 
\[ B,  \quad G_1 = \langle B,w_1 \rangle, \quad  G_2= \langle B,w_2 \rangle, \quad G_5 = \langle B,w_5 \rangle, \quad G_{1, 10} = \langle B, w_1, w_{10} \rangle, \quad G \, .\]

Since each element~$X$ describes for us the decomposition of the corresponding group as a union of double cosets, the inclusions between our six groups are readily worked out. The poset looks like this.
\[ \xymatrix{
  & G \ar@{-}[ld] \ar@{-}[dd] \ar@{-}[rdd] \\
  G_{1,10} \ar@{-}[d]  &     &    \\
  G_1 \ar@{-}[dr]   &  G_2 \ar@{-}[d] & G_5 \ar@{-}[ld] \\
                & B    &      
}
\]

Of course, one may argue that~$S_5$ has only 156 subgroups, and that it may seem easier to go through all of them using GAP and check which of them contain~$B$. We want to argue that the method above would scale well to much larger examples, however. To give a very first idea of how one could optimize the search for intermediate subgroups, we draw a directed graph on~$\{ 0, 1, \ldots, 10 \}$ with an arrow from~$i$ to~$j$ whenever we know the following fact: if~$X \in \A^{01}(\Gamma )$ involves~$T_{w_i}$ and satisfies~$X \odot X = X$, then it must also involve~$T_{w_j}$. For example, we have 
\[ T_{w_9} \odot T_{w_9} = T_{w_0} +  T_{w_5} + T_{w_6} +  T_{w_{10}} \, , \]
so we can place arrows from~$9$ to each of~$0, 5, 6, 10$. On \cref{fig-astuce} we have placed all the arrows obtained from looking at the squares~$T_{w_i} \odot T_{w_i}$, as well as one arrow between~$9$ and~$3$, because there is already an arrow between~$9$ and~$6$, and 
\[ T_{w_9} \odot T_{w_6} = T_{w_3} + T_{w_6} + T_{w_8} + T_{w_9} \, .   \]
Contemplating this figure, we see that if~$X$ involves~$T_{w_i}$ where~$i \in \{ 3,4,6,7,8,9 \}$, and satisfies $X \odot X = X$, then~$X$ must be the sum of all the element of the Coxeter basis (corresponding to the subgroup~$G$). We have reduced the search for nontrivial intermediate subgroups to subsets of~$\{ 0,1,2,5,10 \}$ rather than~$\{ 0,1, \ldots, 10 \}$. Exploiting the graph further (adding more edges), one concludes rapidly. In fact, once the products $T_{w_i} \odot T_{w_j}$ have been computed, the search can be done (and has been done) by hand, with no extra information.

\section{Buildings} \label{sec-buildings}

In this section, the edge-coloured graphs we encounter are initially allowed to be infinite (though we allow only finitely many colours). The basic definitions above still make sense. Quite rapidly, we focus on graphs which are regular for some orders (see Definition~\ref{def-chamb-system-regular}), and this implies that they are locally finite. The last theorem assumes finiteness again.

\subsection{Preliminaries}

We consider buildings as particular edge-coloured graphs (for which vertices are very often called chambers, and paths are very often called galleries, but in this paper we continue to consider these terms as synonymous). Relying on the equivalence of categories given in \cite[Theorem 1.3.1]{charlot}, it is easy to translate any argument or definition given in terms of labelled simplicial complexes, as is the alternative, into the language of edge-coloured graphs. Even so, there are many possible definitions of buildings available, each with its own merits. Here we will have to recall two definitions, rather than just one: the first has inspired the idea of architectures, and the second is needed for the original notion of strong transitivity (and is perhaps more familiar).

We will need a Coxeter system~$(W, S)$ for the discussion, so~$W$ is a group and~$S= \{ s_i : i \in \I \}$ is a set of generating involutions for~$W$, indexed by~$\I$. The order of~$s_i s_j$ will be denoted by~$m_{ij}$; the relations~$(s_is_j)^{m_{ij}}= 1$ constitute a presentation for~$W$, by definition of a Coxeter system. Whenever $f= (i_1, \ldots, i_k)$ is a word on the alphabet~$\I$ -- for example~$f$ might be the type of a gallery in a graph~$\Gamma $ which is coloured by~$\I$ -- we write~$r_f = s_{i_1} \cdots s_{i_k} \in W$. Here and elsewhere the notation follows \cite{weiss} rather closely.

Also useful for our discussion is the Cayley graph~$\cayley(W,S)$, whose vertices are the elements of~$W$, with an edge of colour~$i$ between~$v$ and~$w$ if and only if~$w= s_i v$ (or alternatively~$v= s_i w$). We note that, whenever~$\gamma $ is a gallery in~$\cayley(W, S)$, it is entirely determined by its starting point~$x$ and its type~$f$; indeed if~$f= (i_1, \ldots, i_k)$, then the chambers visited are~$x$, $x s_{i_1}$, $x s_{i_1} s_{i_2}$, $\ldots$, $x r_f$. It follows that an automorphism~$\phi $ of~$\cayley(W, S)$ fixing a vertex~$x$ must be the identity (as it fixes all the galleries starting at~$x$, and~$\cayley(W,S)$ is connected). On the other hand, if~$x, y$ are chambers, then multiplication on the left by~$y x^{-1} \in W$ is an automorphism of the Cayley graph taking~$x$ to~$y$. In the end, the automorphism group of~$\cayley(W,S)$ is identified with~$W$ itself. We say that a type~$f$ is {\em reduced} when any gallery of type~$f$ in~$\cayley(W,S)$ is minimal (that is, realizes the combinatorial distance in~$\cayley(W,S)$ between its endpoints); clearly this needs to be checked only on a single gallery of type~$f$, for the others are obtained by applying automorphisms. Note that, in what follows, we may consider galleries in arbitrary graphs coloured by~$\I$ and ask whether their types are reduced.

Now suppose the edge-coloured graph~$\Gamma $ is a chamber system (Definition~\ref{def-chamb-system-regular}). A first definition is:

\begin{defn} \label{def-bu1}
We say that the chamber system~$\Gamma $ coloured by~$\I$ is a {\em building} with associated Coxeter system~$(W,S)$ when it is endowed with a map 
\[ \delta \colon \Vert(\Gamma ) \times \Vert(\Gamma ) \longrightarrow W  \]
with the following property: for any {\em reduced} type~$f$, and for chambers~$x, y$, we have~$\delta (x, y) = r_f$ if and only if there is a gallery~$\gamma $ in~$\Gamma $, leading from~$x$ to~$y$, whose type is~$f$.

\end{defn}

To formulate the second definition, recall that an induced subgraph of~$\Gamma $ is the edge-coloured graph obtained by selecting a subset of vertices and keeping the edges between them which are present in~$\Gamma $. An {\em apartment} in~$\Gamma $ (of type~$(W,S)$) is an induced subgraph isomorphic to~$\cayley(W,S)$. We may state:

\begin{defn} \label{def-bu2}
 We say that the chamber system~$\Gamma $ coloured by~$\I$ is a {\em building} with associated Coxeter system~$(W,S)$ when: \begin{enumerate}
\item for any two chambers of~$\Gamma $, there is an apartment containing both;
  \item for any chamber~$x$ and apartment~$A$ containing~$x$, there is a ``folding'' map~$\rho_{x, A} \colon \Gamma \longrightarrow A$ which is the identity on~$A$. We also ask that, whenever~$A'$ is another apartment containing~$x$, the restriction of~$\rho_{x, A} $ to~$A'$ an isomorphism (that is, it is the unique isomorphism~$A' \longrightarrow A$ fixing~$x$).
\end{enumerate}

\end{defn}

Here it is meant that~$\rho = \rho _{x, A}$ is a homomorphism of edge-coloured graphs, so if~$x \sim_i y$ in~$\Gamma $, we have either~$\rho (x) \sim_i \rho (y)$ or~$\rho (x) = \rho (y)$.

\begin{lem}
The two definitions of building are equivalent.
\end{lem}

\begin{proof}
  In this proof, $\Gamma $ is a chamber system coloured by~$\I$. If~$\Gamma $ satisfies \cref{def-bu1}, then from Corollary 8.6 and Propositions 8.17, 8.18 and 8.19 from \cite{weiss}, it also satisfies \cref{def-bu2}.

  Now assume that~$\Gamma $ satisfies \cref{def-bu2}, and let~$x,y$ be chambers. Select an apartment~$A$ containing both, and let~$\phi \colon \cayley(W,S) \longrightarrow A$ be an isomorphism. We may as well assume that~$\phi= \phi_{x, A} $ is the unique isomorphism taking~$1 \in W$ to~$x$ (by precomposing with an automorphism if necessary). Let~$w \in W$ be the vertex~$\phi^{-1}(y)$, and put~$\delta (x,y) = w$. Before we even check that this is well-defined, we note that for any type~$f$, reduced or not, such that~$w= r_f$, we can consider the unique gallery~$\gamma $ of type~$f$ in~$\cayley(W,S)$ leading from~$1$ to~$w$, and then~$\phi (\gamma )$ is a gallery of type~$f$ from~$x$ to~$y$ in~$\Gamma $.

  To see that this~$w$ is well-defined, suppose we had chosen another apartment~$A'$ containing~$x$ and~$y$. The unique isomorphism~$\alpha \colon A' \longrightarrow A$ sending~$x$ to itself must be the restriction of~$\rho_{x, A}$. The latter is the identity on~$A$, so~$\alpha (y) = y$. As the composition~$\alpha \circ \phi_{x, A'}$ must be~$\phi_{x, A}$ by uniqueness, we draw~$\phi_{x, A}^{-1}(y) = \phi_{x, A'}^{-1}(y)$. So~$w$ is well-defined. 

  We claim that, whenever~$\gamma $ is a gallery of reduced type~$f$ between~$x$ and~$y$ in~$\Gamma $, and~$A$ is an apartment containing~$x$, the gallery~$\rho_{x, A}(\gamma )$ also has type~$f$ (or in simpler terms, the folding~$\rho_{x, A}$ does not contract any edge of~$\gamma $). This is certainly true of galleries of length~$1$, for~$\rho_{x, A}^{-1}(x) = \{ x \}$. Now suppose~$x, y, \gamma $ and~$f$ constitute a counter-example to the claim with the length of~$\gamma $ minimal. Suppose~$\gamma = (x, x_1, \ldots, x_{k-1}, y)$, with type~$f= (i_1, \ldots, i_k)$.  Since there exists an apartment~$A$ containing~$x$ with~$\rho_{x, A}(\gamma )$ not having type~$f$, then this holds for {\em any} such apartment, and we may as well assume that~$y \in A$. By minimality of~$k$, the gallery~$(x, x_1, \ldots, x_{k-1})$ is mapped by~$\rho_{x, A}$ to a gallery  of type~$f'= (i_1, \ldots, i_{k-1})$ in~$A$; and since~$\phi_{x, A}(y) = y$, we draw~$\phi_{x, A}(x_{k-1}) = \phi_{x, A}(y)$. As a result, we have~$\delta (x, y) = r_{f'} = s_{i_1} \cdots s_{i_{k-1}}$. Interchanging the roles of~$x$ and~$y$, and of~$x_1$ and~$x_{k-1}$, we draw similarly~$\delta (x, y) = s_{i_2} \cdots s_{i_k}$. As~$\delta $ is well defined, we have 
\[ s_{i_1} \cdots s_{i_{k-1}} = s_{i_2} \cdots s_{i_k} \, ,   \]
so that~$r_f = s_{i_1}s_{i_2} \cdots s_{i_k} = s_{i_2} \cdots s_{i_{k-1}}$ (using that~$s_{i_1}^2 = 1$). Expressing~$r_f$ as a product of~$k-2$ generators is absurd, however, as~$f$ is a reduced word of length~$k$. We have proved the claim.

In particular, suppose there exists a gallery~$\gamma $ of reduced type~$f$ from~$x$ to~$y$ in~$\Gamma $; then, by the claim just established, we may assume that~$\gamma $ lies entirely in an apartment~$A$ containing both~$x$ and~$y$. It is then clear that~$\delta (x, y) = r_f$. We have finally established the properties required for \cref{def-bu1}.
\end{proof}

Now we feel free to quote results about buildings from any of the usual sources. The following improves our observations about galleries of reduced type, in the above proof.

\begin{prop} \label{prop-weiss-galleries}
Let~$\Gamma $ be a building, and let~$\gamma $ be a gallery of type~$f$ between~$x$ and~$y$. Then~$\gamma $ is minimal if and only if~$f$ is reduced. Moreover, if~$\gamma $ is minimal, then it is the unique gallery of type~$f$ between~$x$ and~$y$. If~$A$ is an apartment containing~$x$ and~$y$, and if~$\gamma $ is minimal, then~$\gamma $ lies entirely in~$A$.
\end{prop}

\begin{proof}
We quote from \cite{weiss}. The first statement is (ii) of Proposition 7.7, and the second is (iii) of the same proposition. The last statement is obvious at this point (using our observation in the previous proof), and it is also Corollary 8.9 in {\em loc.\ cit.}
\end{proof}

\begin{ex} \label{ex-generalised-digon}
  If~$\Gamma $ is a complete graph (on a single colour), then it is a building with Coxeter group~$S_2 = C_2 = \{ 1, s \}$, with~$\delta (x, x)= 1$ and~$\delta (x, y)= s$ when~$x \ne y$. The apartments are pairs of vertices, with the edge between them, and the folding~$\rho_{x, A}$, when~$\Vert(A)= \{ x, x' \}$, is defined by~$\rho_{x, A}(x)= x$ and~$\rho_{x, A}(y) = x'$ when~$y \ne x$.


  \topnextpage{\imageright{.5}{\begin{figname} \label{fig-digon}
A building with Coxeter group~$C_2\times C_2$
\end{figname}}{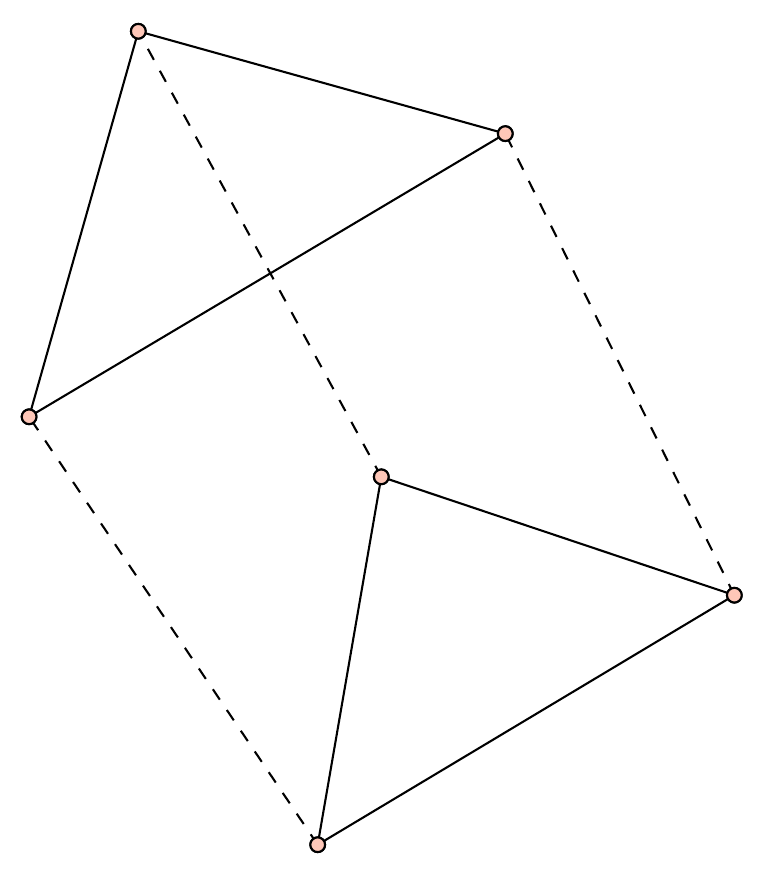}}
  
The product of two buildings is again a building. For example, taking a complete graph on three vertices, with edges of colour~$1$, and a complete graph on two vertices, with edges of colour~$2$, we obtain the graph pictured on \cref{fig-digon}. The Coxeter group here is~$W= C_2\times C_2 = \langle s_1, s_2 : s_1^2 = s_2^2= (s_1s_2)^2 = 1 \rangle$, and the apartments are squares with edges of alternating colours. The folding maps are easy to imagine.

Buildings obtained as the product of two complete graphs are called {\em generalised digons} and they play a special role in the theory (namely, the role of the least interesting chamber systems on two colours). 
\end{ex}

\subsection{Iwahori-Hecke algebras}

We fix a Coxeter system~$(W, S)$ as above.

\begin{defn}
  For each~$i \in \I$, let~$q_i$ be a complex number. The {\em Iwahori-Hecke algebra } associated to $(W,S, (q_i)_{i \in \I})$ is the algebra $\IH(W,S, (q_i)_{i \in \I})$ generated over~$\C$  by elements~$T_i$ for~$i \in \I$, subject to the following relations, for~$i,j \in \I$:
\begin{equation} \label{eq-hecke1}
(T_i - q_i)(T_i+1) = 0 \, ,
\end{equation}
\begin{equation} \label{eq-hecke2}
T_i T_j T_i \cdots  = T_j T_i T_j \cdots
\end{equation}
with~$m_{ij}$ terms on either side.
\end{defn}

\begin{prop} \label{prop-iwahori-hecke}
In $\IH(W,S, (q_i)_{i \in \I})$, we put $T_f = T_{i_1} T_{i_2} \cdots T_{i_k}$ when~$f= (i_1, \ldots, i_k)$ is a type. With this notation, when~$f$ is reduced, the element~$T_f$ depends only on~$r_f \in W$, and we may call it~$T_{r_f}$. Moreover, the various elements~$T_w$ thus obtained, for~$w \in W$, form a generating family for~$\IH(W,S, (q_i)_{i \in \I})$.
\end{prop}

\begin{proof}
  When two reduced words~$f$ and~$g$ satisfy~$r_f = r_g$, then \cite{weiss} says that we can obtain~$g$ from~$f$ by a series of ``elementary homotopies'', which consist precisely in replacing a sequence~$p(i,j)= (i,j, i, j, \ldots, )$ (of length~$m_{ij}$) by~$p(j,i)$. So~$T_f = T_g$ in the algebra.

  Moreover, when~$f$ is not reduced, then it is homotopic (in the above sense) to a type~$g$ involving a repetition~$(i,i)$ for some~$i$: indeed in~\cite{weiss} this is taken as the definition of ``reduced'', while Proposition 4.3 in {\em loc.\ cit.\ } gives the equivalence with our definition.  Replacing~$T_i^2$ by~$(q_i - 1) T_1 + q_i$, we can express~$T_g=T_f$ as a sum of monomials of the form~$T_{f'}$ with~$f'$ shorter than~$f$. An obvious induction allows us to conclude.
\end{proof}

\begin{rmk} \label{rmk-keep-parkinson-quiet}
We caution that the proposition does not claim that the elements~$T_w$ form a basis for the algebra in question. One can show that this holds when the numbers~$(q_i)_{i \in \I}$ have the property that~$q_i = q_j$ whenever~$s_i$ and~$s_j$ are conjugate elements of~$W$ (for lack of a reference stating exactly this fact, we point out~\cite[\S7.1]{hump},~\cite[\S6.1]{gar}, from which it can be deduced without too much work). Moreover, when the~$q_i$'s are the orders in a regular building~$\Gamma $, which is the case of chief interest for us, one can show that this condition is automatically fullfilled: see~\cite{parkinson} . Our approach does not rely on these results, however, and this is why our proof of the next theorem is a little different than that proposed by Parkinson in~\cite{parkinson}. 
\end{rmk}


\begin{thm} \label{thm-adjacency-building-is-hecke}
  Let~$\Gamma $ be building with associated Coxeter system~$(W,S)$, and suppose that~$\Gamma $ is regular with orders~$(q_i)_{i \in \I}$. Then the adjacency algebra~$\A(\Gamma )$ and the Iwahori-Hecke algebra $\IH(W,S, (q_i)_{i \in \I})$ are isomorphic (as algebras with distinguished generators~$(T_i)_{i \in \I}$). Moreover, the elements~$T_w$ for~$w \in W$ form a basis for either algebra.
\end{thm}

\begin{proof}
  Let~$T_i \in \A(\Gamma )$ denote the adjacency operator for the colour~$i$, and write~$T_f= T_{i_1} \cdots T_{i_k}$ whenever $f= (i_1, \ldots, i_k)$ is a type. We claim that, when~$f$ is reduced and~$x$ is a vertex of~$\Gamma $, we have 
\[ x \cdot T_f(x) = \sum_{y:\delta (x,y) = r_f} \, y \, .  \tag{*}  \]
Indeed, from \cref{lem-action-monomial}, to evaluate~$x \cdot T_f(x)$ we must inspect the galleries of type~$f$ starting from~$x$; however, by \cref{prop-weiss-galleries}, such a gallery is the unique one of its type between its endpoints. Hence the identity (*), keeping in mind that a gallery of type~$f$ exists between~$x$ and~$y$ if and only if~$\delta (x, y) = r_f$, by definition of a building.

It follows now blatently that~$T_f$ depends only on~$r_f$, when~$f$ is reduced. In the sequel we may call it~$T_{r_f}$. 
It particular, \cref{eq-hecke2} holds in~$\A(\Gamma )$. (It is a classical fact about Coxeter groups that~$(i,j,i,j, \ldots)$ with~$m_{ij}$ terms is a reduced type.) Of course \cref{eq-hecke1} also holds, as we know from \cref{ex-complete-graph} and the following discussion. Thus we have a homomorphism~$\IH(W,S, (q_i)_{i \in \I}) \longrightarrow \A(\Gamma )$, mapping~$T_i$ to~$T_i$, which must be surjective. Note that the elements called~$T_w$ for~$w \in W$ in \cref{prop-iwahori-hecke} map to the elements with the same name in~$\A(\Gamma )$.

To show that this homomorphism is an isomorphism, and at the same time that the elements~$T_w$ for~$w \in W$ form a basis, it remains to prove that they are linearly independent in~$\A(\Gamma )$. This is, however, obvious: fixing a vertex~$x$, we see that the vectors~$x \cdot T_w$, by the relation (*), involve disjoint sets of vertices as~$w$ varies.
\end{proof}

\begin{coro} \label{coro-buildings-have-architecture}
$\Gamma $ has a canonical architecture.
\end{coro}

\begin{proof}
Put~$\bar \delta (x, y) = T_w$ when~$\delta (x, y) = w \in W$. We need to check that~$\bar \delta $ is an architecture. Axiom (Ar1) is trivial, given the theorem; axiom (Ar2) is the identity (*) observed during the proof.
\end{proof}

\subsection{Strongly transitive actions}

\begin{thm} \label{thm-strong-actions-equiv-conditions}
  Let~$\Gamma $ be a finite building with associated Coxeter system~$(W,S)$, and suppose that~$\Gamma $ is regular with orders~$(q_i)_{i \in \I}$. Assume that~$G$ is a finite group acting on~$\Gamma $. Then the three conditions below are equivalent:

  \begin{enumerate}
  \item $G$ acts transitively on the pairs~$(x, A)$ where~$x$ is a chamber and~$A$ is an apartment containing~$x$,
  \item $G$ acts vertex-transitively and~$\End_G(V) = \A(\Gamma )$,
    \item $G$ acts vertex-transitively, and given an arbitrary vertex~$x_0$, the $\stab_G(x_0)$-orbits on~$\Vert(\Gamma )$ are exactly the spheres centered at~$x_0$, for the canonical architecture~$\delta $.
\end{enumerate}


\end{thm}

\begin{proof}
Throughout the proof we have a group~$G$ acting vertex-transitively, so we may identify~$\Vert(\Gamma )$ with~$G/B$, after choosing a vertex~$x_0$ and letting~$B$ denote its stabilizer. We make a general remark: let~$g \in G$, let~$w= \delta (\bar 1, \bar g) \in W$, then we have 
\[ B \bar g \subset \{ x \in \Vert(\Gamma ) : \delta (\bar 1, x) = w \} \, . \tag{$\dagger$}
\]
Indeed, the distance~$\delta $ is~$\Aut(\Gamma )$-invariant by construction, and~$B$ acts by automorphisms, fixing~$\bar 1$. Condition (3) expresses that ($\dagger$) is an equality (for all~$g\in G$).


So we start by assuming (1), and we show (3) first. Let~$x_1$ and~$x_2$ satisfy~$\delta (\bar 1, x_1) = \delta (\bar 1, x_2) = w$, and let~$A_i$ be an apartment through~$\bar 1$ and~$x_i$, for~$i= 1,2$. By assumption there is~$b \in B$ such that~$b(A_1) = A_2$. However, in~$A_2$ there is just one chamber at distance~$w$ from~$\bar 1$, namely~$x_2$, so~$b(x_1) = x_2$. This show that the right hand side of ($\dagger$) makes up one $B$-orbit, so that ($\dagger$) is an equality, and we have (3).

Next we show $(3) \Longrightarrow (2)$.  Indeed, in \S\ref{subsec-double-cosets} we have defined, for~$g \in G$, the operator~$\phi_g \in \End_G(V)$, and we have 
\[ \bar 1 \cdot \phi_{g^{-1}} = \sum_{x \in B\bar g} x = \sum_{x : \delta (\bar 1, x)= w} x = \bar 1 \cdot T_w  \, ,   \]
where~$w=  \delta (\bar 1, \bar g)$. This shows that~$\phi_{g^{-1}} = T_w\in \A(\Gamma )$, by~$G$-equivariance. The various operators~$\phi _g$, for~$g \in G$, generate~$\End_G(V)$ by \cref{prop-lux}, so~$\End_G(V) = \A(\Gamma )$. We have (2).

We turn to $(2) \Longrightarrow (3)$, which is rather similar.  In the notation above, it remains true that 
\[ \bar 1 \cdot \phi_{g^{-1}}   =   \sum_{x \in B\bar g} x   \]
and 
\[ \bar 1 \cdot T_w  = \sum_{x : \delta (\bar 1, x)= w} x \, .   \]
Given the inclusion ($\dagger$) when~$w= \delta (\bar 1, \bar g)$, and the fact that~$\phi_{g^{-1}}$ can be expressed as linear combination of the various~$T_v$ for~$v \in W$, it must be the case that ($\dagger$) is an equality (as above, it follows that~$\phi_{g^{-1}} = T_w$). Thus (3) holds.

Finally, we assume that (3) holds, and show (1).  So let~$A_1$ and~$A_2$ be two apartments containing~$\bar 1$, and let us show that there is~$b \in B$ such that~$b(A_1) = A_2$. Here it will be important to note that, the building being finite, it is certainly {\em spherical} (that is, $W$ is finite): in this case, in each apartment containing~$\bar 1$ there is a unique chamber~$y$ which is {\em opposite} to~$\bar 1$, in the sense that it maximizes the distance to~$\bar 1$ (computed within the apartment, or within~$\Gamma $: the two are the same by \cref{prop-weiss-galleries}). See \cite{weiss}. So let~$y_i \in A_i$ be the opposite of~$\bar 1$, for~$i= 1, 2$, and let us pick~$b \in B$ such that~$b(y_1) = y_2$: this~$b$ exists from ($\dagger$) and the fact that~$\delta (\bar 1, y_1) = \delta (\bar 1, y_2)$ (this common distance, as an element of~$W$, is the opposite of~$1 \in W$ in~$\cayley(W,S)$).

If~$x \in A_1$ is any chamber, then it lies on a minimal gallery~$\gamma $ in~$A_1$ from~$\bar 1$ to~$y_1$, by~\cite[Proposition 5.4]{weiss} . So~$b(\gamma )$ is a minimal gallery between~$\bar 1$ and~$y_2$, and it must lie entirely in~$A_2$ by \cref{prop-weiss-galleries}. In particular~$b(x) \in A_2$, so~$b(A_1) \subset A_2$, and since the apartments have the same finite number of chambers, we conclude that~$b(A_1) = A_2$.
\end{proof}

\begin{rmk} \label{rmk-two-archs-on-buildings}
Of course  when all three conditions above hold, the architecture obtained from \cref{coro-buildings-have-architecture} is the same as that obtained from \cref{thm-strong-trans-implies-arch}, from the uniqueness statement in this theorem. 
\end{rmk}

\subsection{Example: projective planes} \label{subsec-projective-planes}

\topnextpage{
\noindent \begin{minipage}{.5\textwidth}
\image{1}{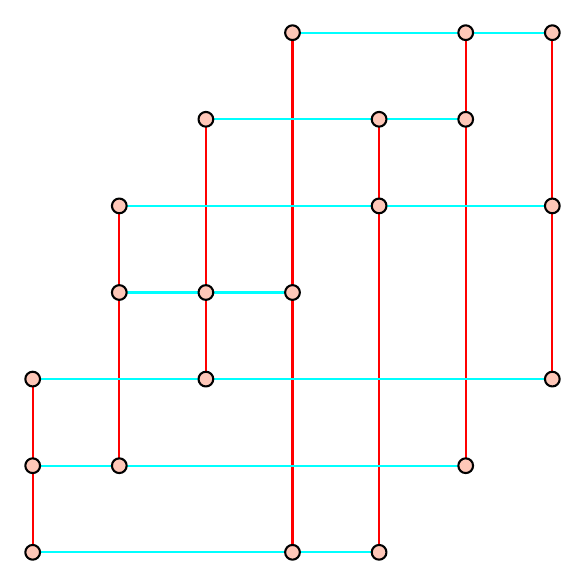}
\end{minipage}~\begin{minipage}{.5\textwidth}
\image{1}{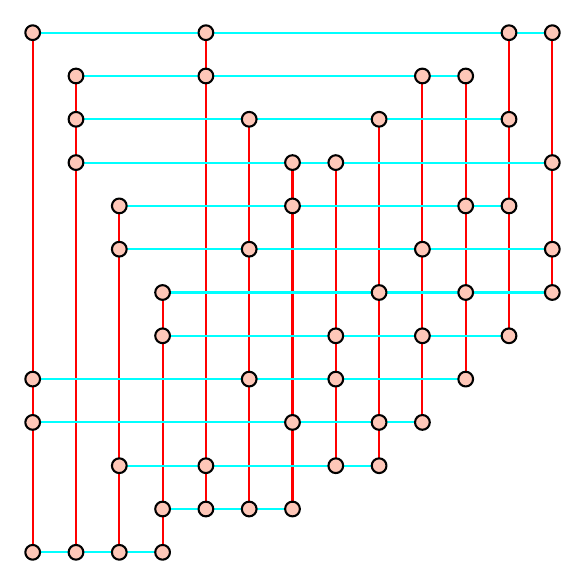}
\end{minipage}

\begin{figname} \label{fig-proj-planes}
The chamber systems of~$\P^2(\f_2)$ (left) and~$\P^2(\f_3)$ (right).
\end{figname}
}

The definition of a projective plane was given in the Introduction. \cref{fig-proj-planes} provides a couple of examples. To avoid trivialities, we take it as part of the definition that each point is incident with at least three lines, and each line is incident with at least three points. (See~\cite[Theorem I, 3.5]{ueber} for a proof that allowing ``three'' to be replaced by ``two'' would only add a couple of rather uninteresting projective planes to the list.) Under this assumption, it is easy to show that, given a finite projective plane, there is an integer~$q \ge 2$ called its {\em order}, such that each point is incident with~$q+1$ lines, and each line is incident with~$q+1$ points. In what follows, $\P =(P,L)$ is a finite projective plane, of order~$q$, and we study~$\Gamma = \chamber(\P)$, its chamber system. The edge-coloured graph~$\Gamma $ is thus regular with orders~$q_1= q_2 = q$. It has~$(q^2 + q+1)(q+1)$ vertices (there are~$q^2+q+1$ points, and so~$(q^2 +q+1)(q+1)$ pairs of incident points and lines).

Theorem 2.2.9 in~\cite{buek} shows that projective planes are the same thing as the so-called generalised 3-gons (or generalised triangles), while~\cite[\S7.14]{weiss} shows that generalised 3-gons are in one-to-one correspondence, via the chamber system construction, with buildings whose associated Coxeter group is~$W=S_3$ (of order~$6$), with its generators~$s_1 = (12)$ and~$s_2 =(23)$.  Thus our graph~$\Gamma $ is such a building.

\cref{thm-adjacency-building-is-hecke} applies, showing that~$\A(\Gamma )$ is the algebra generated by~$T_1$ and~$T_2$, with~$(T_i - q)(T_i+1) = 0$ for~$i=1, 2$, end with $T_1 T_2 T_1 = T_2 T_1 T_2$. It has dimension~$6=|W|$, with Coxeter basis $1, T_1, T_2, T_1T_2, T_2T_1, T_1T_2T_1$. If we start with the graph on the left of \cref{fig-proj-planes} and write down its adjacency operators~$T_1$ and~$T_2$, then we may compute the generic element 
\[ aI + bT_2 + cT_1 + dT_2T_1 + eT_1T_2 + fT_1T_2T_1  \]
and obtain the matrix displayed in the Introduction.

The easiest way to construct a projective plane of order~$q$, when~$q$ is a power of a prime, is of course to consider the~$1$-dimensional subspaces of~$\f_q^3$ as points, and the~$2$-dimensional subspaces as lines (identifying such a plane with the set of lines contained in it). We call this example the {\em Desarguesian} projective plane of order~$q$ (a very common notation is~$\P^2(\f_q)$). Now, the difficult result by Ostrom and Wagner \cite{twotrans} states that, if a group~$G$ acts on~$\P$, and if~$G$ is~$2$-transitive on the set of points, then~$\P$ is Desarguesian.

However, we note that:

\begin{lem} \label{lem-strong-2-trans}
Let the group~$G$ act on~$\P$, and so also on~$\Gamma $. If the action on~$\Gamma $ is strongly transitive, then the action on the set of points is 2-transitive. In fact, the action on the set of non-degenerate, ordered triangles is transitive.
\end{lem}

\noindent (Of course, saying that we have an action of a group~$G$ on~$\P$ means that~$G$ acts on~$P$, preserving the set of lines.)

\begin{proof}
  Vertex-transitivity on~$\Gamma $ clearly implies that~$G$ is transitive on points. Let~$p$ be a point, and let~$H= \stab_G(p)$. We show 2-transitivity on points first, so we show that, given points~$p_1$ and~$p_2$ such that~$p, p_1, p_2$ are distinct, there exists~$h \in H$ such that~$h(p_1) = p_2$.

  Indeed, for~$i=1, 2$ let~$\ell_i$ be the line through~$p$ and~$p_i$, and let~$x_i = (p_i, \ell_i)$. Let~$\ell$ be a line through~$p$ which is distinct from~$\ell_1$ and~$\ell_2$, and let~$x= (p, \ell)$. Then~$\delta (x, x_i) = T_2 T_1$ (to go from~$x$ to~$x_i$, first cross an edge of colour~$2$ to get to~$(p, \ell_i)$, then an edge of colour~$1$ to reach~$(p_i, \ell_i) = x_i$). If~$B = \stab_G(x)$, it follows from \cref{thm-strong-actions-equiv-conditions} that there exists~$b \in B$ with~$b(x_1) = x_2$. In particular~$b(p_1) = p_2$, and as~$B \subset H$, we are done with 2-transitivity in points.

Suppose now that~$(p_1, p_2, p_3)$ and~$(p_1', p_2', p_3')$ are given non-degenerate triangles. From the first part, we may as well assume that $p_1 = p_1'$ and $p_2 = p_2'$, and we look for an element of~$G$ fixing~$p_1$ and~$p_2$ while taking~$p_3$ to~$p_3'$. Consider the flags $x= (p_1, (p_1p_2))$, $y= (p_3, (p_3p_2))$ and $y'= (p_3', (p_3'p_2))$. Then $\delta (x, y) = \delta (x, y') = T_1 T_2 T_1$ (equivalently, there exists a gallery of type~$(1,2,1)$ from~$x$ to either~$y$ or~$y'$). From \cref{thm-strong-actions-equiv-conditions} again, there exists~$b \in B = \stab_G(x)$ such that~$b(y) = y'$. This element~$b$ takes~$p_3$ to~$p_3'$, takes $(p_3p_2)$ to~$(p_3'p_2)$, fixes~$p_1$, and fixes the line~$(p_1p_2)$; it follows that it fixes~$p_2$, which is the intersection of~$(p_1p_2)$ and~$(p_3p_2)$.
\end{proof}

\begin{rmk}
As announced in the Introduction, one can show conversely that, if~$G$ acts transitively on non-degenerate triangles, then the action is strongly transitive. We leave this as an exercise.
\end{rmk}

Applying the Ostrom-Wagner theorem, we get:

\begin{coro}
If~$\Gamma $ is strongly transitive, then $\P$ is Desarguesian.
\end{coro}

We will provide a direct proof of the corollary, see \cref{thm-moufang-proj-desargues}.

Now we turn to the study of analogous results for affine planes. Note that affine planes are not buildings -- if nothing else, because the dimension of the adjacency algebra will be proved to be odd.

\section{Affine planes} \label{sec-affine}

\subsection{Definitions}

Again, we refer to the Introduction for the definition of an affine plane. See \cref{fig-affine-fano} below for a picture. From now on we assume that each point of an affine plane is incident with at least two lines, and that each line is incident with at least two points. Given a finite affine plane, one shows easily that there is an integer~$q \ge 2$, called its {\em order}, such that each point is incident with~$q+1$ lines, and each line is incident with~$q$ points. In what follows, we let~$\AP = (P,L)$ denote a finite affine plane of order~$q$, and we study its chamber system~$\Gamma  = \chamber(\AP)$. It is regular with orders~$q_1 = q-1$ and~$q_2= q$, and it has~$q^2(q+1)$ vertices.

Given an integer~$q$ which is a prime power, the {\em Desarguesian affine plane} of order~$q$ is obtained by considering the elements of~$\f_q^2$ as points, with the usual affine lines.

\subsection{Relative positions}

Let~$x= (p, \ell)$ and~$y= (p',\ell')$ be two flags of~$\AP$ (vertices of~$\Gamma $). We shall study their relative positions, identifying seven basic situations. In each case, we define an element~$\delta (x,y) \in \A(\Gamma )$. Of course, it is important to make sure that the seven possibilities are mutually exclusive, but this will be obvious.

The first trivial case is when~$x=y$; we put~$\delta (x,y) = I$ in this case. Next, if~$\ell=\ell'$ but~$p \ne p'$, that is when~$x \sim_1 y$, we put~$\delta (x,y) = T_1$. Similarly when~$x \sim_2 y$, which happens when~$p=p'$ but~$\ell \ne \ell'$, we put~$\delta (x,y) = T_2$.

The remaining cases are more interesting:

\begin{itemize}
\item Suppose~$p\ne p'$, $\ell \ne \ell'$, but~$p'$ is on the line~$\ell$. This happens if and only if there is a gallery of type~$(1,2)$ between~$x$ and~$y$, and this gallery is then unique. We put~$\delta (x, y) = T_1T_2$ and for short, we write~$T_{12} = T_1 T_2$.

\item Suppose~$p\ne p'$, $\ell \ne \ell'$, $p'$ is not on~$\ell$, and the lines~$\ell$ and~$\ell'$ intersect in~$p$. This happens if and only if there is a gallery of type~$(2,1)$ between~$x$ and~$y$, and this gallery is then unique. We put~$\delta (x, y) = T_2T_1$ and we write~$T_{21} = T_2 T_1$.

\item Suppose~$p\ne p'$, $\ell \ne \ell'$, $p'$ is not on~$\ell$, and the lines~$\ell$ and~$\ell'$ intersect in~$p'' \ne p$. This happens if and only if there is a gallery of type~$(1,2,1)$ between~$x$ and~$y$, and this gallery is then unique. We put~$\delta (x, y) = T_1T_2T_1$ and we write~$T_{121} = T_1T_2 T_1$. A supplementary remark is that in this case, there is also a gallery of type~$(2,1,2)$ between~$x$ and~$y$, but this does not characterize the relative position, as we see with the next and final case.

  \item Suppose~$p\ne p'$, $\ell \ne \ell'$, $p'$ is not on~$\ell$, and the lines~$\ell$ and~$\ell'$ are parallel. This happens if and only if there is a gallery of type~$(2,1,2)$, but no gallery of type~$(1,2,1)$, between~$x$ and~$y$. The gallery of type~$(2,1,2)$ is then unique. We put~$\delta (x, y) = T_2 T_1 T_2 - T_1 T_2 T_1$ and we write~$\ttot = T_2 T_1 T_2 - T_1 T_2 T_1$. (It seems that the notation~$T_{212}$ should be kept for~$T_2 T_1 T_2$.)
\end{itemize}

Now we put 
\[ \W= \big\{ I, T_1, T_2, T_{12}, T_{21}, T_{121}, \ttot \big\} \, .   \]
We point out that~$\W$ does not depend on the choice of particular vertices.

\begin{lem} \label{lem-almost-arch-aff}
\begin{enumerate}
\item For~$T \in \W$ and~$x$ an arbitrary vertex, we have
\[ x \cdot T = \sum_{y : \delta (x, y) = T} \, y \, .   \]
In other words, axiom (Ar2) holds.

\item Each of the seven situations actually occurs. In fact, for~$T \in \W$, and for an arbitrary vertex~$x$, we can find~$y$ such that~$\delta (x, y) = T$. As a result, the elements of~$\W$ are linearly independent.
\end{enumerate}
\end{lem}

\begin{proof}
The point (1) is obvious in the first six cases by an application of \cref{lem-action-monomial}. For~$\ttot$, we write the equation as 
\[ x \cdot T_2 T_1 T_2 = x \cdot T_{1} T_2 T_1 + \sum_{y : \delta (x, y) = \ttot} \, y \, .   \]
Now replace $x \cdot T_2 T_1 T_2$ and $x \cdot T_{1} T_2 T_1$ by the expression given in \cref{lem-action-monomial}, and the identity becomes obvious by the discussion above.

(2) follows from the fact that each line has at least two points, and each point is on at least two lines.
\end{proof}

\topnextpage{\imageright{.75}{\begin{figname} \label{fig-affine-fano}
      The chamber system of the Desarguesian affine plane of order~$2$.
\end{figname}
}{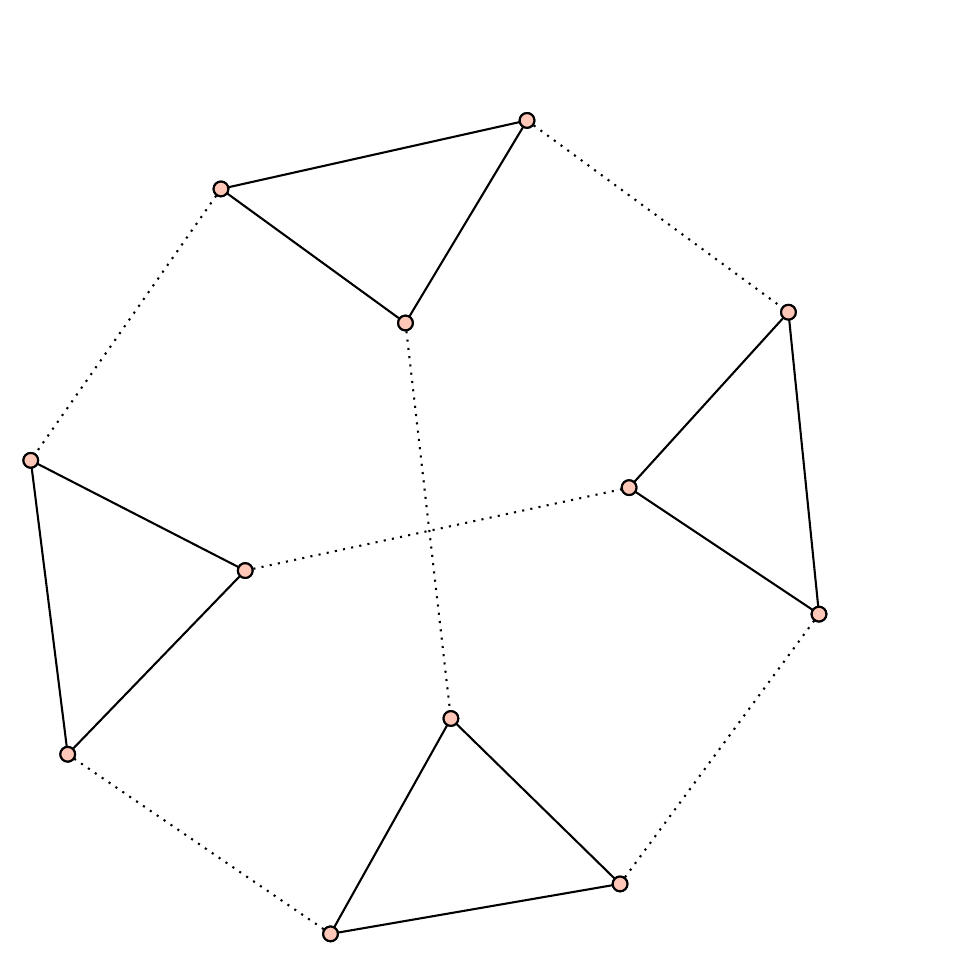}
}

We have almost proved that~$\delta $ is an architecture, with~$\W$ as the associated Coxeter basis. What is missing, of course, is a proof that~$\A(\Gamma )$ has dimension~$7$. We turn to this.

\subsection{Description of the incidence algebra}

\begin{defn}
  Let~$q$ be a complex number. We define~$\Aff(q)$ to be the algebra over~$\C$ generated by~$T_1$ and~$T_2$ subject to

  \begin{equation} \label{eq-aff-basic}
 (T_1 - (q-1)I)(T_1+I) = 0 \, , \quad (T_2 - qI)(T_2 + I) = 0 \, , 
\end{equation}
\begin{equation} \label{eq-aff-q}
(T_1T_2)^2 = (q-1)T_2T_1 + (q-1)T_2T_1T_2 - T_1T_2T_1 \, , 
\end{equation}
\begin{equation} \label{eq-aff-q-2}
(T_2T_1)^2 = (q-1)T_1T_2 + (q-1)T_2T_1T_2 - T_1T_2T_1 \, . 
\end{equation}

\end{defn}

\begin{prop} \label{prop-adj-for-affine-planes}
Let~$\Gamma $ be the chamber system of an affine plane of order~$q$. Then~$\A(\Gamma )$ is isomorphic to~$\Aff(q)$, and has dimension~$7$. The map~$\delta $ is an architecture on~$\Gamma $, with associated Coxeter basis~$\W$.
\end{prop}

\begin{proof}
  First we show that there is a homomorphism~$\Aff(q) \longrightarrow \A(\Gamma )$ mapping~$T_i$ to~$T_i$, and for this we need to prove that the defining relations for~$\Aff(q)$ hold in~$\A(\Gamma )$. The first two are a consequence of \cref{ex-complete-graph}. We prove \cref{eq-aff-q}, and \cref{eq-aff-q-2} will follow by taking transposes.

  Using the notation above, we want to show that 
\[ (T_1T_2)^2 = (q-1) T_{21} + (q-1) \ttot + (q-2)T_{121} \, . \tag{*}  \]
We pick a vertex~$x$ and compare the effect of either side, when applied to~$x$. On the left hand side we have 
\[ x \cdot (T_1T_2)^2 = (x \cdot T_{121}) \cdot T_2 = \sum_{y : \delta (x,y) = T_{121}} \, T_2(y) = \sum_{y : \delta (x,y) = T_{121}} \sum_{z : z \sim_2 y } \, z\, ,   \]
using \cref{lem-almost-arch-aff}. If~$x= (p, \ell)$ and~$y$ is such that~$\delta (x, y) = T_{121}$, then~$y= (p', \ell')$ with~$p' \ne p$, $\ell' \ne \ell$, and the intersection of~$\ell$ and~$\ell'$ is~$p'' \ne p$. There are three possibilities for a vertex~$z$ with~$y \sim_2 z$, although we always have~$z= (p', \ell'')$ with~$\ell'' \ne \ell'$ :
\begin{itemize}
\item the line~$\ell''$ may be~$(pp')$, the line through~$p$ and~$p'$. In this case~$\delta (x, z) = T_{21}$. Conversely, if we start with~$z= (p', (pp'))$, the number of vertices~$y= (p', \ell')$ which satisfy at the same time~$\delta (x, y) = T_{121}$ and~$y \sim_2 z$ is~$q-1$: one is free to choose the line~$\ell'$ among the lines joining~$p'$ with a point of~$\ell$ different from~$p$.

\item the lines~$\ell$ and~$\ell''$ may be parallel. In this case~$\delta (x, z) = \ttot$. Starting from~$z=(p', \ell'')$ with~$\delta (x, z) = \ttot$, there are~$q-1$ vertices~$y= (p', \ell')$ which satisfy at the same time~$\delta (x, y) = T_{121}$ and~$y \sim_2 z$: again, one has the same choices for~$\ell'$ as in the previous case. (This time, the line~$\ell'$ must avoid~$p$ not because we must have~$\ell' \ne \ell''$, but because~$\delta (x, y) = T_{121}$.)

  \item the lines~$\ell$ and~$\ell''$ may intersect in a point~$r$ distinct from both~$p$ and~$p''$.  In this case~$\delta (x, z) = T_{121}$. Starting from~$z=(p', \ell'')$ with~$\delta (x, z) = T_{121}$, so that~$\ell''$ and~$\ell$ intersect at a point~$r \ne p$, we find this time~$q-2$ vertices~$y= (p', \ell')$ which satisfy at the same time~$\delta (x, y) = T_{121}$ and~$y \sim_2 z$: here~$\ell'$ may be any line through~$p'$, intersecting~$\ell$ at a point~$p''$ which is distinct from~$p$ and from~$r$.
\end{itemize}
In the end, we have 
\[ x \cdot (T_1T_2)^2 = (q-1) \sum_{\delta (x, z)= T_{21} }  z \quad + \quad  (q-1) \sum_{\delta (x, z) = \ttot}  z \quad + \quad (q-2) \sum_{\delta (x, z) = T_{121}}  z \, .   \]
Using \cref{lem-almost-arch-aff}, and since~$x$ is arbitrary, we have (*).

We have thus proved that~$\A(\Gamma )$ is a quotient of~$\Aff(q)$. Given the form of the relations, moreover, it is clear that the elements of~$\W$ are a generating family for~$\A(\Gamma )$, and likewise, the elements of~$\Aff(q)$ defined by analogous formulae are a generating family for~$\Aff(q)$. However, we know from \cref{lem-almost-arch-aff} that these are linearly independent elements. It follows that~$\W$ is a basis for~$\A(\Gamma )$, and that the homomorphism $\Aff(q) \longrightarrow \A(\Gamma )$ is an isomorphism.

The axiom (Ar1) has just been checked, and we see that~$\delta $ is an architecture.
\end{proof}

\begin{ex}
If we consider the example given on \cref{fig-affine-fano} and compute the adjacency operators~$T_1$ and~$T_2$, the generic element 
\[ aI + bT_2 + cT_1 + dT_2T_1 + eT_1T_2 + fT_1T_2T_1 + g(T_2T_1T_2 - T_1T_2T_1)  \]
is given by 
\[ \left(\begin{array}{rrrrrrrrrrrr}
a & b & b & g & f & d & c & e & e & g & d & f \\
b & a & b & f & g & d & d & g & f & e & c & e \\
b & b & a & e & e & c & d & f & g & f & d & g \\
g & f & d & a & b & b & g & d & f & c & e & e \\
f & g & d & b & a & b & e & c & e & d & g & f \\
e & e & c & b & b & a & f & d & g & d & f & g \\
c & e & e & g & d & f & a & b & b & g & f & d \\
d & g & f & e & c & e & b & a & b & f & g & d \\
d & f & g & f & d & g & b & b & a & e & e & c \\
g & d & f & c & e & e & g & f & d & a & b & b \\
e & c & e & d & g & f & f & g & d & b & a & b \\
f & d & g & d & f & g & e & e & c & b & b & a
\end{array}\right) \, .   \]
In other words, the matrices of this form, with~$a, b, c, d, e, f, g \in \C$ as parameters, form an algebra isomorphic to~$\Aff(2)$.
\end{ex}

\subsection{Representation theory}

\begin{prop}
  There are three homomorphisms~$\varepsilon_i \colon \A(\Gamma ) \longrightarrow \C$, for~$i=1,2,3$, taking~$(T_1, T_2)$ to~$(-1, -1)$, $(q-1, -1)$ and $(q-1, q)$ respectively. Also, there is a homomorphism $\rho \colon \A(\Gamma ) \longrightarrow M_2(\C)$ with 
\[\rho (T_1) = \left(\begin{array}{rr}
-1 & 0 \\
0 & q - 1
\end{array}\right) \, , \quad \rho (T_2) = \left(\begin{array}{rr}
q - 1 & q \\
1 & 0
\end{array}\right) \, .   \]
The modules afforded by these homomorphisms are irreducible. What is more, any irreducible module for~$\A(\Gamma )$ is isomorphic to one and only one of them.
\end{prop}

\begin{proof}
  From \cref{eq-aff-basic},  we see that any homomorphism~$\A(\Gamma ) \longrightarrow \C$ must take~$T_1$ to~$-1$ or~$q-1$, and~$T_2$ to~$-1$ or~$q$. This leaves four possibilities, and we need to check the remaining relations. We find that~$(-1, q)$ is impossible, and the other three combinations work.

  The proposed formulae (which were found by trial and error) for~$\rho $ do define a~$2$-dimensional representation. We caution that \cref{eq-aff-q} and \cref{eq-aff-q-2} must be both checked, since the~$2\times 2$-matrices above are not symmetric. Performing the check is straightforward though, as is the verification that no~$1$-dimensional subspace is invariant under the action, so~$\rho $ is irreducible.

  Since~$\A(\Gamma )$ is semisimple (\cref{prop-adj-algebra-is-semisimple}), non-commutative, and of dimension~$7$, we see immediately that 
\[ \A(\Gamma ) \cong M_2(\C) \times \C^3 \, ,   \]
so it cannot have more irreducible modules than the ones presented here.
\end{proof}

\begin{rmk}
Suppose we worked with~$\Aff(q)$ rather than~$\A(\Gamma )$, where~$q$ is now {\em any} complex number. The modules defined in the proposition can still be considered, and if we merely assume that~$q \ne 0$, then they are all irreducible. This shows that~$\Aff(q)/J$ has dimension at least~$2^2 + 1 + 1 + 1 = 7$, where~$J$ is the radical of~$\Aff(q)$. However it is clear that~$\Aff(q)$ always has dimension~$\le 7$, so we conclude that, when~$q \ne 0$, the algebra~$\Aff(q)$ is semisimple, of dimension~$7$ (thus generalizing from those cases when there exists an affine plane of order~$q$).
\end{rmk}

\cref{prop-criterion-same-V} now predicts that the~$\A(\Gamma )$-module~$V(\Gamma )$ depends only on~$q$, and not on the particular affine plane of order~$q$ chosen to build~$\Gamma $. We can confirm this by direct computation:

\begin{lem} \label{lem-V-as-adj-module}
Let~$n_0, n_1, n_2, n_3$ be integers so that
\[ V(\Gamma ) \cong n_0 \rho + n_1 \varepsilon _1 + n_2 \varepsilon_2 + n_3 \varepsilon_3 \, ,   \]
as an~$\A(\Gamma )$-module. Then~$n_0= q^2 - 1$, $n_1 = (q-1)^2(q+1)$, $n_2 = q$, and $n_3= 1$.
\end{lem}

\noindent Of course here we abuse the notation slightly by writing~$\varepsilon_i$ or~$\rho $ for the modules afforded by these homomorphisms.

\begin{proof}
It is enough to show that~$n_0, n_1, n_2, n_3$ are solutions of the following system: 
\[ \left\{ \begin{array}{lllllllcl}
  2n_0 & + & n_1 & + & n_2 & + & n_3 & = & q^2 (q+1) \\
  n_0 &  &  &  &  & + & n_3 & = & q^2 \\
  n_0 & + & n_1 & + & n_2 &  &  & = & q^3 \\
  n_0 &  &  & + & n_2 & + & n_3 & = & q(q+1) \\
  n_0 & + & n_1 &  &  &  &  & = & q (q^2 - 1)
\end{array}\right.  \]

The first equation is obtained by comparing dimensions.

For the second and third equations, we look at~$V(\Gamma )$ as a module equipped with the sole action of~$T_2$. Since~$\Gamma $ is a chamber system, when we delete the edges of colour~$1$, we are left with~$q^2$ copies of the complete graph~$K_{q+1}$ on~$q+1$ vertices (with all its edges of the colour~$2$). Hence~$V(\Gamma ) = q^2 V(K_{q+1})$, as a~$T_2$-module. What is more, from \cref{ex-complete-graph}, we know that~$V(K_{q+1}) = L_q + q L_{-1}$, where~$L_\alpha $ is~$1$-dimensional with~$T_2$ acting by multiplication by~$\alpha $, for~$\alpha \in \{ -1, q \}$. In the end~$V(\Gamma ) = q^2 L_q + q^3 L_{-1}$. On the other hand, the module~$L_q$ occurs once in~$\rho $ and once in~$\varepsilon_3$, and not in~$\varepsilon_1$ or~$\varepsilon_2$, whence~$q^2 = n_0 + n_3$. Looking at~$L_{-1}$ instead, we obtain~$q^3 = n_0 + n_1 + n_2$ for similar reasons.

Analyzing~$V(\Gamma )$ under the sole action of~$T_1$, we obtain the third and fourth equation in a similar way. (The fact that~$n_3 = 1$ can also be seen as a consequence of the connectedness of~$\Gamma $; we leave this as an exercise.)
\end{proof}

\begin{coro} \label{coro-restrictions-strong-trans-action-affine}
Suppose the group~$G$ acts strongly transitively on~$\Gamma $. Then~$G$ has irreducible representations~$\tau_0, \tau_1, \tau_2$ of degree~$q^2 - 1$, $(q-1)^2(q+1)$ and~$q$ respectively. It follows that~$|G|$ is divisible by~$q^2(q-1)^2(q+1)$.
\end{coro}

\begin{proof}
We know from \S\ref{subsec-double-cosets} that the irreducible~$G$-modules occuring in~$V$ are in correspondence with the irreducible~$\End_G(V)$-modules, with the multiplicities and the dimensions exchanged. Since we assume now that~$\End_G(V) = \A(\Gamma )$, we know the multiplicities of the irreducible~$\End_G(V)$-modules in~$V$ from the lemma. Hence the corollary is just a translation.

In particular, as it is classical that the dimension of an irreducible representation of~$G$ must divide~$|G|$, the order of the group is divisible by~$n_1 = (q-1)^2 (q+1)$. Also, the action is transitive on the set of~$q^2(q+1)$ vertices, so~$|G|$ is also divisible by~$q^2$. These numbers are relatively prime, so we are done.
\end{proof}

\begin{coro} \label{coro-eigenvalues}
Let~$\Gamma_0$ be the simple graph obtained from~$\Gamma $ by forgetting the colours of the edges. Then the adjacency eigenvalues of~$\Gamma_0$ are~$-2$ with multiplicity $(q-1)^2(q+1)$, then~$q-2$ with multiplicity~$q$, as well as $2q-1$ with multiplicity~$1$, and finally 
\[ \frac{2q - 3 \pm \sqrt{4q+1}}{2} \, ,   \]
each with multiplicity~$q^2 -1$.
\end{coro}

\begin{proof}
The adjacency matrix of~$\Gamma_0$ is~$T_1 + T_2$. Under~$\varepsilon_1, \varepsilon_2$, or $\varepsilon_3$, the operator~$T_1 + T_2$ acts by multiplication by~$-2$, $q-2$ or~$2q-1$ respectively. On the other hand 
\[ \rho (T_1 + T_2) = \left(\begin{array}{rr}
  q-2 & q \\
  1 & q-1
\end{array}\right) \, .  \]
The eigenvalues of this matrix are $\frac{2q - 3 \pm \sqrt{4q+1}}{2}$. It remains to work out the multiplicities, but these are exactly given by the lemma.
\end{proof}

\subsection{Desarguesian planes}

\begin{thm} \label{thm-moufang-affine-desargues}
Let~$\AP$ be an affine plane of finite order~$q$. Then~$\AP$ is Desarguesian if and only if there exists a group~$G$ acting on~$\AP$, in such a way that the induced action on~$\Gamma = \chamber(\AP)$ is strongly transitive.
\end{thm}

\begin{proof}
  We start with the easy half. Assume~$\AP$ is the usual Desarguesian affine plane on~$\f_q^2$, and pick~$G = \f_q^2 \rtimes \operatorname{GL}_2(\f_q)$. Of course~$G$ acts on~$\AP$ and so also on~$\Gamma $, and it is obvious that the action is vertex-transitive (=flag-transitive). Let~$e_1, e_2$ be the canonical basis for~$\f_q^2$, let~$x_0 = (0, \ell_1)$ where~$\ell_1 = \langle e_1 \rangle$, and let~$B$ denote the stabilizer of~$x_0$. To show that~$\End_G(V) = \A(\Gamma )$, we must check that~$\dim \End_G(V) = | B \bs G /B| = \dim \A(\Gamma ) = 7$. In other words, we must count the orbits of~$B$ on~$G/B$.

  These orbits are easily described. One is~$\{ x_0 \}$. All the~$(0, \ell)$ where~$0 \in \ell$ but~$\ell \ne \ell_1$ constitute one orbit, as do the~$(p, \ell_1)$ with~$p \in \ell_1$, $p \ne 0$, and also the~$(p, \ell)$ with~$p \ne 0$, $p \in \ell_1 \cap \ell$, $\ell \ne \ell_1$. The flags~$(p, \ell)$ with~$p \not\in \ell_1$ break into three orbits: the orbit of~$(e_2, e_2 + \langle e_1 \rangle)$, that of~$(e_2, e_2 + \langle e_2 \rangle)$, and finally that of~$(e_2, e_2 + \langle e_1+e_2 \rangle)$. In the end there are~$7$ orbits, and this proves that the action is strongly transitive.

  Now assume conversely that~$G$ acts on~$\AP$, and that the induced action on~$\Gamma $ is strongly transitive. We use the classification of linear spaces announced in~\cite{moult} and proved in a series of papers, culminating in~\cite{saxl}. The main result classifies the pairs~$(G, S)$, where~$S$ is a linear space and~$G$ acts flag-transitively on it, into two families: (I) a certain finite list described below, and (II) a class of pairs for which the order of~$G$ divides~$q^2(q^2 - 1)a$, where~$q$ is the order of the plane and~$q^2 = p^a$ for a prime~$p$. We can immediately see that our pair~$(G, \AP)$ is not of type (II), for \cref{coro-restrictions-strong-trans-action-affine} tells us that~$|G|$ is divisible by~$q^2(q-1)^2(q+1)$; we would have ~$q-1 | a$, which is easily seen to be impossible.

  Thus we explore the list (I), which of course contains the Desarguesian affine planes, and we must exclude all the other affine planes from that list. Here we follow~\S3.2 in~\cite{moult}, and the notation~$G_0$ will denote the stabilizer of a point, so that~$G = T \rtimes G_0$, where~$T$ is the translation subgroup, of order~$q^2$. \cref{coro-restrictions-strong-trans-action-affine} tells us that~$|G_0|$ is divisible by~$(q-1)^2(q+1)$.

  One candidate is the Hering plane of order~$q= 27$. In this case~$G_0= \operatorname{SL}_2(\f_{13})$, which has order~$2184$, and this is not divisible by~$(q-1)^2(q+1) = 18928$. Thus the action cannot be strongly transitive.

  Next we treat the case of the Lüneburg planes. These have order~$q= Q^2$ where~$Q= 2^{2e+1}$, and~$G_0$ is a subgroup of~$\Aut(Suz(Q))$ where~$Suz(Q)= {}^2B_2(Q)$ is the Suzuki group, of order~$(Q^2 + 1)Q^2(Q-1)$. The only outer automorphisms of the Suzuki groups are the field automorphisms of~$\f_Q$, so~$\Out(Suz(Q))$ is cyclic of order~$2e+1$. So now we know that~$(q-1)^2(q+1) = (Q^2 - 1)^2(Q^2 + 1)$ divides~$(Q^2 + 1)Q^2(Q-1)(2e+1)$. We deduce that~$(Q+1)^2(Q-1) | Q^2(2e+1)$, and this is impossible (in fact since~$Q+1$ and~$Q-1$ are odd, we deduce that $(Q+1)^2 (Q-1) | 2e + 1$). Thus Lüneburg planes are excluded.

  The hardest case is that of the ``nearfield plane'' $\AP$ of order~$q=9$. Assuming some group~$G$ acts strongly transitively on it, then the same can be said of~$\Aut(\AP)$, so we pursue with~$G= \Aut(\AP)$. Here the group~$G= T \rtimes G_0$ is described in~\cite[\S5]{foulser} : in fact, explicit~$4 \times 4$ matrices are given, which generate the group~$G_0$. Using Sage/GAP, we compute that the order of~$G_0$ is~$3840$, which is indeed divisible by~$(q-1)^2(q+1) = 640$. So we must work a bit harder. Or rather, we let the computer do this for us: we build~$G$ and ask the GAP library to compute the degrees of its irreducible characters: these are 1, 4, 5, 6, 10, 15, 16, 20, 24, 80, 160, 240, 320 (obtained in a matter of seconds). However, if the action were strongly transitive, the group~$G$ would have representations of degree  $(q-1)^2(q+1) = 640$ and~$q=9$, from \cref{coro-restrictions-strong-trans-action-affine}. This shows that the action is not strongly transitive.

  The only remaining affine planes on the list (I) are the Desarguesian planes, and we are done.
\end{proof}

As promised in the Introduction, we can give a ``projective'' version, with a very similar proof:

\begin{thm} \label{thm-moufang-proj-desargues}
Let~$\P$ be a projective plane of finite order~$q$. Then~$\P$ is Desarguesian if and only if there exists a group~$G$ acting on~$\P$, in such a way that the induced action on~$\Gamma = \chamber(\P)$ is strongly transitive.
\end{thm}

\begin{proof}
  This is so similar to the affine case that an outline will suffice (notice that the argument will be a little easier; recall also that it is best to deduce this from the Ostrom-Wagner theorem, but we give this proof to emphasize that very similar arguments can be used in both the projective and the affine case). Of course the ``if'' statement is the more deserving of a proof. The algebra~$\A(\Gamma )$ was described in~\S\cref{subsec-projective-planes}. It has two $1$-dimensional representations, say~$\varepsilon_1$ and~$\varepsilon_2$, with~$T_1, T_2$ acting both as~$-1$ or both as~$q$; there is also a~$2$-dimensional, irreducible representation~$\rho $, with 
\[ \rho (T_1) = \left(\begin{array}{rr}
q & 0 \\
0 & -1
\end{array}\right) \, , \quad \rho (T_2) = \left(\begin{array}{rr}
\frac{-1}{q + 1} & \frac{-q}{q + 1} \\
\frac{-q^{2} - q - 1}{q + 1} & \frac{q^{2}}{q + 1}
\end{array}\right) \, . 
  \]
  There are no other irreducible representations. The module~$V$ splits as the direct sum~$n_0 \rho + n_1 \varepsilon_1 + n_2 \varepsilon _2$. If we examine the action of~$T_1$ alone, as in the previous proof, we find 
\[ n_0 + n_1 = (q^2+q+1)(q+1) \, , \quad n_0 + n_2 = q^2 + q + 1 \, .   \]
However, we can see directly that~$n_2 = 1$: an element of~$V$, seen as a function on~$\Vert(\Gamma )$, which is a $q$-eigenvector for both~$T_1$ and~$T_2$, must be constant on all the $1$-residues (= the connected components of the graph obtained by keeping only the edges of one colour), so it is constant overall, by connectedness. We draw $n_0 = q(q+1)$ and $n_1 = q^3$.

As in the previous proof, we deduce that if a group~$G$ acts strongly transitively on~$\Gamma $, then its order is divisible by~$q^3(q^2+q+1)(q+1)$. In the classification, this excludes all the line spaces of type (II). Among those of type (I), the only projective planes are the Desarguesian planes, so we are done.
\end{proof}

\subsection{Clique planes}

As promised in the Introduction, we also treat the case of the so-called {\em clique planes}. Since this is rather similar to our study of affine planes, but somewhat easier, we will be a little sketchy. We should perhaps add that the rank~$2$ geometries which are prominently of interest to the experts (say, in the study of sporadic groups) are generalized polygons (which are buildings), affine planes, the Peterson geometry, and clique planes. This motivates the inclusion of the material below, so that the basic cases would be covered in this paper. (The only serious omission is perhaps that of tilde geometries.)

For each~$q \ge 2$, the clique plane of order~$q$ is thus simply the complete graph on~$q+2$ points, seen as a linear line space, and we study its chamber system~$\Gamma $. It may be described as the edge-coloured graph on the vertices~$(i,j)$ for~$1 \le i , j \le q+2$ and~$i \ne j$, with an edge of colour~$1$ between~$(i,j)$ and~$(j,i)$, and an edge of colour~$2$ between~$(i,j)$ and~$(i,k)$ for~$k \ne j$. It is regular with orders~$q_1 = 1$ and~$q_2= q$. We let~$G= S_{q+2}$, which obviously acts on~$\Gamma $, and we shall see that the action is strongly transitive.

Pick a vertex~$x_0 = (i_0, j_0)$ and let~$B= \stab_G(x_0) \cong S_q$. We put~$\delta (x, x)=I$ and otherwise classify the vertices~$y \ne x$ according to their~$B$-orbit:

\begin{itemize}
\item One~$B$-orbit is~$\{ y \}$ where~$y= (j_0, i_0)$. We put~$\delta (x, y) = T_1$.
\item The set $\{ (i_0, k) : k \ne i_0, j_0 \}$ is a~$B$-orbit, and it is comprised of the~$2$-neighbours of~$x$. For~$y$ in this orbit, we put~$\delta (x, y) = T_2$.
\item The set $\{ (j_0, k) : k \ne i_0, j_0 \}$ is a~$B$-orbit, comprised of the vertices at the end of a gallery of type $(1,2)$ starting from~$x$. We put~$\delta (x, y) = T_{12} :=T_1T_2$ for such a vertex~$y$.
\item The set~$\{ (k, i_0) : k \ne i_0, j_0\}$ is also a~$B$-orbit; for~$y$ in this orbit, we let~$\delta (x, y) = T_{21} := T_2 T_1$, for obvious reasons.
\item The set~$\{ (k, j_0) : k \ne i_0, j_0 \}$ is a~$B$-orbit, and we put~$\delta (x, y) = T_{121}:= T_1 T_2 T_1$ for~$y$ in this orbit.
  \item Finally, the set~$\left\{ (k, \ell) : \{ k, \ell \} \cap \{ i_0, j_0 \} = \emptyset \right\}$ is a $B$-orbit, and for~$y$ in this orbit, we put~$\delta (x, y)= T_{212^*} := T_2 T_1 T_2 - T_1 T_2 T_1$, for reasons similar to the above in the case of affine planes.
\end{itemize}

We put 
\[ \W= \big\{ I, T_1, T_2, T_{12}, T_{21}, T_{121}, \ttot \big\} \, .   \]

\begin{thm}
  Let~$\CP$ be the clique plane of order~$q$. The adjacency algebra of~$\CP$ is the algebra~$\circle(q)$ generated by~$T_1$ and~$T_2$, subject to 
\[ (T_1 - I)(T_1+I) = 0 \, , \quad (T_2 - qI)(T_2+I) = 0 \, ,   \]
as well as 
\[ (T_2T_1)^2 = T_1T_2 + T_2T_1T_2 - T_1T_2T_1 \, ,   \]
and 
\[ (T_1T_2)^2 = T_2T_1 + T_2T_1T_2 - T_1T_2T_1 \, .   \]
It has dimension~$7$. The map~$\delta $ above is an architecture, with associated Coxeter basis~$\W$. The action of~$S_{q+2}$ is strongly transitive.
\end{thm}

\begin{proof}
If we argue as we have done for affine planes, we see that everything follows if we can only prove that
\[ (T_1T_2)^2 = T_2T_1 + T_2T_1T_2 - T_1 T_2 T_1 \, ,   \]
or equivalently that 
\[ x \cdot T_{121} T_2 = x \cdot T_{21}+ x \cdot T_{212^*} \tag{*} \]
for an arbitrary vertex~$x$. On the left hand side, this expands to 
\[ \sum_{y : \delta (x, y) = T_{121}} \sum_{z : \delta (y, z)=T_2} z \, ,   \]
from (Ar2), which is easily established. Here, if~$x= (i_0, j_0)$, then~$z= (k, \ell)$ with~$\ell \ne k$ and~$\ell \ne j_0$. We have two possibilities. We may have~$\ell \ne i_0$, in which case~$\delta (x, z) = T_{212^*}$ ; on the other hand, for any~$z$ with~$\delta (x, y) = T_{212^*}$, we find a unique~$y$ such that~$\delta (x, y) = T_{121}$ and~$\delta (y, z) = T_2$, trivially. The second possibility is that~$\ell = i_0$, so that~$\delta (x, z) = T_{21}$ : and again, starting from such a~$z$, there is a unique~$y$ with $\delta (x, y) = T_{121}$ and~$\delta (y, z) = T_2$. And so, using (Ar2) again, we do indeed have the relation (*).
\end{proof}

\bibliography{myrefs}
\bibliographystyle{custom}

\end{document}